\documentclass[11pt]{amsart}
\usepackage{amscd,amssymb}
\usepackage{amsfonts}
\usepackage{amsmath}
\usepackage{graphicx} 
\usepackage{color}
\usepackage{hyperref} 

\newcommand\lie[1]{\mathfrak{#1}}
\newcommand\qbin[2]{\left[  \begin{array}{c}  #1 \\ #2  \end{array}  \right]}

\newcommand{\CC}{\mathbb{C}}

\newcommand{\Z}{\mathbb{Z}}

\newcommand{\Q}{\mathbb{Q}}

\newcommand\on[1]{\rm{#1}}

\newcommand{\A}{\mathcal A}

\newcommand{\Sb}{\mathbb{S}}
\newcommand{\cL}{\mathcal{L}}

\newtheorem{theorem}{Theorem}[section]
\newtheorem{corollary}[theorem]{Corollary}
\newtheorem{proposition}[theorem]{Proposition} 
\newtheorem{lemma}[theorem]{Lemma} 
 
\newtheorem{definition}[theorem]{Definition}

\title[Tensor powers]{Tensor powers of vector representation of $U_q(\lie{sl}_2)$ at even roots of unity.}

\author{Anna Lachowska}
\address{A.L.: EPFL CH-1015 Lausanne,
Switzerland}
\email{anna.lachowska@epfl.ch}
\author{Olga Postnova}
\address{O.P.: Euler International Mathematical Institute, Laboratory of Mathematical Problems of Physics, St. Petersburg Department of Steklov Institute of Mathematics, Saint Petersburg, Fontanka river emb. 27,
191023 Saint Petersburg,
Russia}
\email{postnova.olga@gmail.com}
\author{Nicolai Reshetikhin}
\address{N.R.: Yau Mathematical Sciences Center, Tsinghua University, Beijing, China; BIMSA, Beijing, China; Department of Mathematics, UC Berkeley, Berkeley, CA, USA; Saint Petersburg State University, Saint Petersburg, Russia}
\email{reshetik@math.berkeley.edu}
\author{Dmitry Solovyev}
\address{D.S.: Yau Mathematical Sciences Center, Tsinghua University, Beijing, China}
\email{dimsol42@gmail.com}
\begin{document}

\dedicatory{Dedicated to the memory of A. M. Vershik.}

\maketitle

\begin{abstract}
We study the decomposition of tensor powers of two dimensional irreducible representations of quantum $\mathfrak{sl}_2$ at even roots of unity into direct sums of tilting modules. We derive a combinatorial formula for multiplicity of tilting modules in the $N$-th tensor power of two dimensional irreducible representations, interpret it in terms of lattice paths and find its asymptotic behavior when $N\to\infty$. We also describe the limit of character and Plancherel measures when $N\to\infty$. We consider both $U_q(\mathfrak{sl}_2)$ with divided powers and the small quantum $sl_2$.
\end{abstract}

\tableofcontents

\section{Introduction} 
\subsection{} A typical problem in asymptotical representation theory, the subject pioneered by A. M. Vershik, is the study of infinite sequences of branching rules. For example, if $A_n$, $n\in\mathbb{Z}_{>0}$ is a sequence of associative algebras, $V_n$ is a sequence of finite dimensional\footnote{Finite dimensionality is not necessary. For example in \cite{F} similar problem was analyzed for tensor powers of the fundamental representation of $\hat{sl_2}$ where all representations are infinite dimensional.} representations of $A_n$ and $B_n\subset A_n$ is a sequence of subalgebras, consider restrictions of $V_n$ to $B_n$. We have:
$$
V_n\lvert_{B_n}\cong\bigoplus_j \left(V^{(j)}_{B_n}\right)^{\oplus m_j(n)}
$$
where $V^{(j)}_{B_n}$ are indecomposable representations of $B_n$. Natural important problem related to such a sequence is to find the asymptotic of $m_j(n)$ as $n\to\infty$ and to find the asymptotic of the Plancherel measure
$$
p_j(n)=\frac{m_j(n)\dim V^{(j)}_n}{\dim V_n}.
$$
For earlier references on the subject see \cite{VK}, \cite{LS}, \cite{K}, \cite{B}, \cite{TZ}.

An important example of such type of problems is when $H$ is a Hopf algebra, $A_n=H^{\otimes n}$, $B_n=H$, $\Delta^{(n)}:H\to H^{\otimes n}$ and $V_n=V^{\otimes n}$ where $V$ is a representation of $H$. When $H=U(\mathfrak{g})$ such problem was studied in \cite{K} for $\mathfrak{g}=\mathfrak{sl}_n$ and $V=\mathbb{C}^n$. It was extended to all simple $\mathfrak{g}$ with any irreducible $V$ in \cite{B} \cite{TZ}. The asymptotic of character measures, which is a deformation family of Plancherel measures was studied in \cite{PR1} \cite{PR2}. In \cite{CEO} a related problem of finding the asymptotical behavior of $b_n=\sum_j m^{(j)}_n$ in the context of tensor categories was addressed.

\subsection{} In this paper we extend the study of statistics of irreducible components in tensor products of finite dimensional representations of simple Lie algebras \cite{K,B,TZ,PR1} to a similar question about tensor powers of $L(1)$, the two dimensional irreducible representation of $U_q(\mathfrak{sl}_2)$ of type ${\bf 1}$, when $q$ is a root of unity.\footnote{Throughout the paper we always imply a complex root of unity.} In this case representation theory is not semisimple and we study the statistics of indecomposable submodules in tensor powers of $L(1)$. We consider two versions of quantum $\mathfrak{sl}_2$ at a root of unity, $U_q(\mathfrak{sl}_2)$ which is the quantized universal enveloping of the $\mathfrak{sl}_2$ with divided powers \cite{Lus89} and its finite dimensional Hopf subalgebra $u_q(\mathfrak{sl}_2)$.

The two dimensional irreducible representation of $U_q(\mathfrak{sl}_2)$ is a simplest indecomposable tilting module. This is why we also use notation $T(1)$ for $L(1)$. Because the category of tilting modules is a tensor subcategory of the category of finite dimensional representations of $U_q(\mathfrak{sl}_2)$, tensor powers of $T(1)$ decompose into direct sums of tilting modules. On indecomposable tilting components of $T(1)^{\otimes N}$ there is a natural probability measure, the character measure. It is a one parametric family of probability measures with the Plancherel measure being a particular case.

The multiplicities of tilting modules over $U_q(\mathfrak{sl}_2)$ in tensor powers of $L(1)$ also give the dimensions of simple modules for the corresponding Temperley-Lieb algebras, which can be considered as endomorphism algebras of tensor powers of $L(1)$ over $U_q(\mathfrak{sl}_2)$, see e.g. \cite{AST}, \cite{ILZ}. We provide a detailed derivation of the multiplicity formulas for the case of $U_q(\mathfrak{sl}_2)$ and the small quantum group $u_q(\mathfrak{sl}_2)$. In the case of the big quantum group $U_q(\mathfrak{sl}_2)$ our result agrees with the formulas obtained earlier in \cite{And17}, where the case of $U(\mathfrak{sl}_2)$ over a field of positive characteristics and the mixed case are also considered.

For $q=e^{i\frac{\pi m}{l}}$ where $l$ is odd and $m$ and $l$ are mutually prime, indecomposable tilting modules of $U_q(\mathfrak{sl}_2)$ of type ${\bf 1}$ are parametrized by integers. For such tilting modules we write $T(k)$ with $k=lk_1+k_0$, $m_1\in\mathbb{Z}_{\geq 0}$, $0\leq k_0\leq l-2$. For details see Section \ref{qsl2}. The multiplicities of $T(k)$ in the decomposition of the tensor power $T(1)^{\otimes N}$
$$
T(1)^{\otimes N}= \bigoplus_{n\in\mathbb{Z}_{\geq 0}}T(n)^{\oplus M^{(l)}_n(N)}
$$
are given in Theorem \ref{comb_coef}, see also \cite{And17}, {\it Corollary 3.5}. They are also computed in \cite{Sol} in terms of lattice paths, and as an integral of characters and their duals in Section \ref{integralchar}. The character of $T(k)$ is determined by characters of the Weyl modules(see Section \ref{qsl2}) and can be explicitly computed (see Lemma \ref{fchar})
\begin{equation*}
	\mathrm{ch}_{T(k_1l+k_0)}(x)=
	\begin{cases}
		\frac{x^{k_0+1}-x^{-(k_0+1)}}{x-x^{-1}},& \textit{if $0\leq k_0\leq l-1$ and $k_1=0$},\\
		\frac{x^{(k_1+1)l}-x^{-(k_1+1)l}}{x-x^{-1}},& \textit{if $k_0=l-1$ and $k_1\in\mathbb{Z}_{>0}$},\\
		(x^{k_0+1}+x^{-(k_0+1)})\frac{x^{k_1l}-x^{-k_1l}}{x-x^{-1}},& \textit{if $0\leq k_0\leq l-2$ and $k_1\in\mathbb{Z}_{>0}$}.
	\end{cases}
\end{equation*}
When $x>0$ all these characters are positive and we have the character probability measure on indecomposable tilting components with the probability of module $T(n)$ occurring in $T(1)^{\otimes N}$ being
\begin{equation}\label{cchh}
	p^{(l)}_N(k;x)=\frac{M^{(l)}_{k}(N)\mathrm{ch}_{T(k)}(x)}{(\mathrm{ch}_{T(1)}(x))^N}.
\end{equation}
In Theorem \ref{theomch} we prove the weak convergence of (\ref{cchh}) to the tensor product $p_1(\bullet;t)\otimes p_2(\bullet;t)$, where $t=\mathrm{ln}x$, $p_2$ is the Gaussian measure on $\mathbb{R}$ with the density
\begin{equation*}
	p_2(\alpha;t)=\sqrt{\frac{l}{2\pi}}\mathrm{cosh}(t)e^{-\frac{\alpha^2l}{2}\left(\mathrm{cosh}(t)\right)^2},\quad\alpha\in\mathbb{R}
\end{equation*}
and $p_1$ is the discrete probability measure on $k_0=0,1,\ldots,l-1$ with
\begin{equation*}
p_1(k_0;t)=
\begin{cases}
	\frac{(e^{t(k_0+1)}+e^{-t(k_0+1)})(e^{t(l-1-k_0)}-e^{-t(l-1-k_0)})}{l(e^{tl}-e^{-tl})},& \textit{if $k_0=0,\ldots,l-2$},\\
	\frac{1}{l},& \textit{if $k_0=l-1$}.
\end{cases}
\end{equation*}
Note that this tensor product of distributions is reminiscent to the tensor product formula for tilting modules from Theorem \ref{tilt_prod}. We use it indirectly in the proof of Theorem \ref{fchar}. We expect a more direct use of Theorem \ref{tilt_prod} in the derivation of the weak limit of the character measure based on Markov processes.

We also compute the intermediate scaling of the character measure(near Plancherel behavior of the character measure), when $N\to\infty$, $t\to0$ such that $t\sqrt{N}$ is finite. And we compute the limit when $N,t\to\infty$ such that $\theta=N e^{-2t}$ is finite, which corresponds to the critical drift. In the latter case the limit probability measure is geometric in the variable $s=\frac{N-k}{2}$.

The character and, as a consequence, the Plancherel measure are special: they are given by running certain Markov process $N$ times. This Markov process is determined by the decomposition of $T(1)\otimes T(k)$. This gives another way to compute the limit of the character measure. We do this for the Plancherel measures for the small quantum $\mathfrak{sl}_2$, for $u_q(\mathfrak{sl}_2)$ in Section \ref{markovuq}. Markov process derivation of the limit of the character measure will be done in a separate publication.

In this paper we focused on $U_q(\mathfrak{sl}_2)$ and $u_q(\mathfrak{sl}_2)$. The extension of these results to other simple Lie algebras and superalgebras involves regular values of coefficients of Kazhdan-Lusztig polynomials and will be done in a separate publication. Following the classical framework via a suitable version of Schur-Weyl duality this study should have implications for the representation theory of Hecke and Brauer algebras, as outlined in e.g. \cite{AST}.

\subsection{} The paper is organized as follows. In Section \ref{one} we review basic facts about quantized universal enveloping algebra $U_q(\mathfrak{sl}_2)$ at even roots of unity. In Section \ref{tp} we derive formula for multiplicities of tilting modules in decomposition of $T(1)^{\otimes N}$ over $U_q(\mathfrak{sl}_2)$. In Section \ref{restUu} we do this for the small quantum group $u_q(\mathfrak{sl}_2)$. In Section \ref{recursion} we give recursive description of the multiplicities. In Section \ref{rec-mult} we describe multiplicity functions in terms of path counting formulae for lattice path models, defined by the given recursive description. In Section \ref{rec-int} we obtain formulae for multiplicities over $U_q(\mathfrak{sl}_2)$ by means of an integral of characters and their duals.

In the remaining sections we study the asymptotic of multiplicities and asymptotic of distributions of tilting modules in $T(1)^{\otimes N}$ in the $N\to\infty$ limit. In Section \ref{probmeas} we define probabilistic measures on indecomposable components in tensor product decomposition of $T(1)^{\otimes N}$ over $U_q(\mathfrak{sl}_2)$ and over $u_q(\mathfrak{sl}_2)$. For the case $U_q(\mathfrak{sl}_2)$ we consider the character measure and the Plancherel measure, for the case $u_q(\mathfrak{sl}_2)$ we consider the Plancherel and the quantum Plancherel measures. In Section \ref{as9} we derive asymptotic of character and Plancherel measures for $U_q(\mathfrak{sl}_2)$. We consider asymptotics in the bulk, the critical drift and the intermediate scaling. In Section \ref{as10} we use Markov processes associated with tensor products to derive asymptotic of Plancherel and quantum Plancherel measures for $u_q(\mathfrak{sl}_2)$. In Appendices \ref{prthu}, \ref{prthu2}, \ref{proofintermsc}, \ref{proofdrift} we give proofs of the main statements.

\subsection{Acknowledgements} The authors are grateful to V. Serganova, P. Shan and H. H. Andersen for inspiring discussions. At earlier stages of this work we enjoyed discussions with A. M. Vershik. This work was supported by Leonhard Euler International Mathematical Institute (agreement no.075-15-2022-289 date 06/04/2022). The work of O.P., N.R. and D.S. was supported by the the RSF grant 21-11-00141. The work of N.R. was also supported by the Collaboration Grant "Categorical Symmetries" from the Simons Foundation, by the grant, by the grant BMSTC and ACZSP (Grant no. Z221100002722017) and by the Changjiang fund. The work of D.S. was supported by the Xing Hua Scholarship program of Tsinghua University.

\section{Algebra $U_q(\lie{sl}_2)$ and its representations}\label{one}
\subsection{Hopf algebra $U_q(\lie{sl}_2)$}\label{qsl2}

Let $v$ be a formal parameter and consider the $\Q(v)$-algebra $U_v(\lie{sl}_2)$ generated by $E, F, K^{\pm1}$ with relations 
\[ K K^{-1} = K^{-1}K =1 , \quad \quad K E = v^2 E K , \quad \quad K F = v^{-2} F K , \] 
\[ EF - F E  = \frac{K - K^{-1}}{v - v^{-1}} .\] 
This is a Hopf algebra with the coproduct, counit and the antipode given by 
\[ \Delta E = E \otimes 1 + K \otimes E , \quad \quad \Delta F = F \otimes K^{-1} + 1 \otimes F, \quad \quad \Delta K^{\pm 1}  = K^{\pm 1} \otimes K^{\pm 1} ,\] 
\[ \varepsilon(E) = \varepsilon(F) = 0, \quad \quad \varepsilon(K^{\pm 1}) = 1  .\]   
\[ S(E) = -K^{-1} E, \quad \quad S(F) = -FK, \quad \quad S(K) = K^{-1} . \]

Set $\A = \Z[v, v^{-1}]$ and denote by $U_\A(\lie{sl}_2)$ the $\A$-subalgebra of $U_v(\lie{sl}_2)$ with $1$  generated by 
\[ E^{(n)} = \frac{E^n}{[n]!},   \quad F^{(n)} = \frac{F^n}{[n]!} \] 
for all $n \in \mathbb{Z}_{\geq 0}$. 
For any $t \geq 0$, $c \in \Z$  let 
\[ \qbin{K;\; c}{t} = \prod_{s=1}^t \frac{K v^{c-s+1} - K^{-1} v^{-c+s-1}}{v^s - v^{-s}}  \in U_\A(\lie{sl}_2) .\]

We have in $U_\A(\lie{sl}_2)$ \cite{Lus89}: 
\[ E^{(n)} E^{(m)} = \qbin{n+m}{m} E^{(n+m)} , \quad \quad F^{(n)} F^{(m)} = \qbin{n+m}{m} F^{(n+m)} .\] 
\[ E^{(p)} F^{(r)} = \sum_{0 \leq t \leq p,r} F^{(r-t)} \qbin{K ; \; 2t-p-r}{t} E^{(p-t)} ,  \quad  p, r \geq 0. \] 

Let $q = e^{\frac{\pi im}{l}}$ be a primitive $2l$-th root of unity, where we assume $l$ to be odd and $m$ and $l$ to be mutually prime. We have $q^l =-1$, $q^{2l} =1$.  Let $U_q(\lie{sl}_2) = U_\A(\lie{sl}_2) \otimes_{\A} \Q[q,q^{-1}]$, where $\A$ acts by $v \rightarrow q$. This is the big quantum group with divided powers specialized at $q$. 
\begin{proposition}  \label{E^lp}
\begin{enumerate} 
\item  Let $p \in \mathbb{Z}_{\geq 0}$. Then we have in $U_q(\lie{sl}_2)$: 
\[ E^{(lp)} = (-1)^{p(p-1)/2} \; \frac{(E^{(l)})^p}{p!}  .\] 
\item Let $m = m_1 l + m_0$, where $0 \leq m_0 \leq l-1$, and $m_1 \in \mathbb{Z}_{\geq 0}$, then 
\[ E^{(lm_1 + m_0)} = (-1)^{m_1(m_1-1)/2+ m_1m_0} \frac{E^{\,m_0}}{[m_0]!} \; \frac{(E^{(l)})^{m_1}}{m_1!}  .\] 
A similar relation holds for $F^{(m_1 l + m_0)}$. 
\item The algebra $U_q(\lie{sl}_2)$ is generated over $\Q(q)$ by $E, E^{(l)}, F, F^{(l)}, K^{\pm 1}$. 
\end{enumerate} 
\end{proposition} 
\begin{proof} 
Part (a) follows from Corollary \ref{q-bin2}, and (b) from Proposition \ref{q-bin1}, then (c) is a direct consequence of (b). 
\end{proof}

\begin{lemma} 
In $U_q(\lie{sl}_2)$ we have $K^{2l}=1$ and $K^l$ is central. 
\end{lemma} 
\begin{proof} 
By definition we have in $U_\A(\lie{sl}_2)$ 
\[ \prod_{s=1}^l (K v^{-s+1} - K^{-1} v^{s-1}) = \prod_{s=1}^l (v^s - v^{-s}) \cdot \qbin{K; \; 0}{l} .\] 
Specializing $v \to q$ we obtain  zero in the right-hand side, since $q^l - q^{-l} = -1 -(-1) =0$, and therefore we have in 
$U_q(\lie{sl}_2)$ 
\[ \prod_{s=1}^l (K q^{-s+1} - K^{-1} q^{s-1}) = K^{-l}  q^{-\frac{l(l-1)}{2}} \prod_{s=1}^l (K^2 - q^{2s-2}) =  K^{-l}  q^{-\frac{l(l-1)}{2}} (K^{2l} -q^{l(l-1)} ) = 0 .\] 
As we have $q^{l(l-1)} = (-1)^{l-1} =1$ since $l$ is odd. Therefore $K^{2l} =1$  in $U_q(\lie{sl}_2)$.  

We have $K^l E^{(p)} = q^{2lp} E^{(p)} K^l = E^{(p)} K^l$ and $K^l F^{(p)} = q^{-2lp} F^{(p)} K^l = F^{(p)} K^l$ for any $p \geq 0$, 
therefore $K^l$ is central in $U_q(\lie{sl}_2)$. 
\end{proof}

\subsection{Weyl, dual Weyl and tilting modules} \label{modules}
There is a unique structure of  $U_v(\lie{sl}_2)$ module on $U_v(\lie{sl}_2)^-$, where $F$ acts by left multiplication, 
$E\cdot 1 =0$, $K \cdot 1 = v^m 1$ for some $m \in \Z$. (The assignment  $K \cdot 1 =  v^m 1$  leads to the definition of modules of type ${\bf 1}$. If we chose $K \cdot 1 = -v^m 1$, we would obtain modules of type ${\bf -1}$, that exhibit similar properties. In what follows we will consider only modules of type ${\bf 1}$). This module is denoted by $Y(m)$ and it has a unique simple quotient $W(m)$. Then we can consider the $U_{\mathcal A}$-modules of $Y_{\mathcal A}(m)$ and $W_{\mathcal A}(m)$, and finally specialize them to $U_q(\lie{sl}_2)$ by a change of base. We will use the same notations $Y(m)$ and $W(m)$ for the obtained $U_q(\lie{sl}_2)$-modules. Let $x$ denote the image of $1$ in $W(m)$, then  $E \cdot x = E^{(l)} \cdot x =0$, and $K x = q^m x$, $\qbin{K ; 0}{l} = \qbin{m}{l} x$. By Corollary \ref{q-bin3} the action of $K$ and $\qbin{K;0}{l}$ determine $m \in \Z$ uniquely, therefore we can say that $x$ has weight $m \in \Z$. 

If $m \geq 0$, then $W(m)$ is finite dimensional of dimension $m+1$ and with character given by the Weyl formula (its weight structure comes from that of the simple module over $U_v(\lie{sl}_2)^-$).

Now we will present the explicit action of the generators of $U_q(\lie{sl}_2)$ in the modules $W(m)$, $m \geq 0$. We can consider the following basis in $W(m)$: 
\[   W(m)  = \{ x, F^{(1)}x,  F^{(2)}x,  \ldots  F^{(m)} x \},  \] 
where $E x = 0, E^{(l)} x =0, Kv = q^m x$.  Using the commutation relations, we derive the following action of the generators $E, E^{(l)}, F, F^{(l)}, K$ on the basis: 
\[ F \cdot F^{(t)} x= [ t+1] F^{(t+1)} x ; \quad \quad  E \cdot F^{(t)} x=  [m-t+1] F^{(t-1)} x \] 
\[ K \cdot  F^{(t)} x =  q^{m-2t} F^{(t)} x, \quad \quad \qbin{K ; 0}{l} F^{(t)} x =  \qbin{m-2t}{l} F^{(t)} x \] 
\[ F^{(l)} \cdot F^{(t)} x = \qbin{l+t}{l} F^{(l+t)} x,  \] 
 \[  E^{(l)} F^{(t)} x =  \qbin{m+l-t}{l} F^{(t-l)} x .  \] 
Here we assume that $F^{(<0)}x = 0, F^{(>m)}x =0$. 
The module $W(m)$ has a unique simple quotient $L(m)$. Therefore, for any $m \in \Z$ there exists a simple module $L(m)$ of type ${\bf 1}$ with  a primitive vector of weight $m$.

We can also define the dual Weyl module by setting $W^*(m) = {\rm Hom}(W(m), \Q(q))$ for any $m \in \mathbb{Z}_{\geq 0}$. In particular, 
we have  that if $Kv = q^k v$ for $v \in W(m)$, then $Kv^* = q^{-k} v$ for $v^* \in W^*(m)$. 

The properties of modules $L(m), W(m), W^*(m)$ are investigated in detail in \cite{Lus89}, \cite{AT} ,\cite{APW1}, section 4. 
Here is a summary of their properties: 
\begin{theorem} \cite{AT}, section 2.  \label{Weyl} 
For each $m \in \mathbb{Z}_{\geq 0}$ there exists a Weyl module $W(m)$, dual Weyl module $W^*(m)$ over $U_q({\lie sl}_2)$ defined above 
that have the following properties: 
\begin{enumerate} 
\item $W(m) \simeq W^*(m) \simeq L(m)$ if and only if either $0 \leq m \leq l-2$, or $m = lm_1 -1$.  
\item If $m = lm_1 + m_0$ with $0 \leq m_0 \leq l-2$ and $m_1 \geq 0$, then we have the short exact sequences of $U_q({\lie sl}_2)$-modules: 
\[ 0 \longrightarrow L(lm_1 - m_0 -2) \longrightarrow W(m)  \longrightarrow L(m) \longrightarrow 0 , \]
\[ 0 \longrightarrow L(m) \longrightarrow W^*(m) \longrightarrow L(m_1 l -m_0 -2) \longrightarrow 0 .\] 
\end{enumerate} 
\end{theorem} 

\begin{definition} A tilting module $T $ (of type ${\bf 1}$) over $U_q(\lie{sl}_2)$ is a finite dimensional module that has a filtration by Weyl modules and a filtration by dual Weyl modules. 
\[   0 \subset  M_1 \subset M_2 \subset ... \subset M_n = T , \quad \quad M_i / M_{i-1} \cong W_{k(i)}  ,  \] 
\[  0 \subset Q_1 \subset Q_2 \subset  ... \subset Q_n = T , \quad \quad Q_i / Q_{i-1} \cong W_{m(i)}^* , \]
for some $k(i)$, $m(i) \in \mathbb{Z}_{\geq 0}$. 
\end{definition}

It follows from the definition that a dual of a tilting module is also tilting. 
\begin{theorem} \cite{AT}, Section 2. \label{tilt_prop1} 
\begin{enumerate} 
\item For each  $k \in \Z_{\geq 0}$ there exists a unique indecomposable tilting module $T(k)$ of type ${\bf 1}$ such that $T(k) : W(k)=  1$ and $T(k) : W(k') =0$ for $k' >k$. These modules form a complete list of inequivalent indecomposable tilting modules of type ${\bf 1}$  over  $U_q(\lie{sl}_2)$. 
\item  A tensor product, a direct sum and a direct summand of tilting modules is a tilting module. 
\end{enumerate} 
\end{theorem}

Thus the tilting modules over $U_q(\lie{sl}_2)$ form an additive category, closed under finite direct sums, finite tensor products and taking the dual. 

\begin{theorem} \cite{AT}, \label{isomLW}
 Proposition 2.20. 
\label{tilt_prop2}
\begin{enumerate} 
\item We have $T(k) \cong W(k) \cong W(k)^* \cong L(k)$  if and only if either $ 0 \leq k \leq l-2$, or $k = lk_1 -1$ for some $k_1 \in \mathbb{Z}_{> 0}$. 
\item For $k = l k_1 + k_0$, $k_1 \in \mathbb{Z}_{> 0}$, $0 \leq k_0 \leq l-2$ we have the short exact sequence of  $U_q(\lie{sl}_2)$-modules
\[  0 \longrightarrow  W(lk_1+k_0 ) \longrightarrow T(lk_1 + k_0) \longrightarrow W(l k_1 -2 - k_0) \longrightarrow 0 .\] 
Since  $T_k^* \simeq T_k$, we also have  
\[  0 \longrightarrow  W^*(lk_1- 2 -k_0 ) \longrightarrow T(lk_1 + k_0) \longrightarrow W^*(l k_1 + k_0) \longrightarrow 0 .\] 
\end{enumerate} 
\end{theorem} 

\begin{corollary}  \label{tilt_char} 
The character of an indecomposable tilting module over $U_q(\lie{sl}_2)$  is 
\[  {\on{ch}}_{T(k)} = {\on{ch}}_{W(k)} , \quad \quad {\rm if} \quad  0 \leq k \leq l-2 \quad {\rm or} \quad k = lk_1-1, \;\; k_1 \in \mathbb{Z}_{> 0}, \]
\[  {\on{ch}}_{T(lk_1 + k_0)} = {\on{ch}}_{W(lk_1 + k_0)} + {\on{ch}}_{W(lk_1 - 2 - k_0)} , \quad \quad  k_1 \in \mathbb{Z}_{> 0}, \;\; 0 \leq k_0 \leq l-2.\] 
\end{corollary}

\subsection{Tensor product theorem for simple modules of type $1$ over $U_q(\lie{sl}_2)$} \label{tpthm}
Let $u_q = u_q(\lie{sl}_2)$ be the subalgebra of $U_q(\lie{sl}_2)$ generated by $E,F, K^{\pm 1}$. 

\begin{proposition}  \label{Lm0} 
For $m \in \Z$ such that $0 \leq m \leq l-1$, consider the simple $U_q(\lie{sl}_2)$ module $L(m)$, and let $x \in L(m)$ 
be a nonzero vector such that $K x = q^m x$ and $Ex =0$. Then 
\begin{enumerate} 
\item $L(m) = u_q \cdot x$.  
\item The restriction of $L(m)$ to $u_q$ is a simple $u_q$ module. 
\item $E^{(l)} , F^{(l)}$ act as $0$ in $L(m)$. 
\end{enumerate} 
\end{proposition} 
\begin{proof} 
This follows from the fact that $W(m) \simeq L(m)$ for $0 \leq m \leq l-1$ by Theorem \ref{Weyl} and the explicit definition of $W(m)$. From the same definition we have that $E^{(l)}$ and $F^{(l)}$ act by $0$ on any vector in $L(m)$. 
\end{proof} 

Let $\mathcal{U} = \mathcal{U}(\lie{sl}_2)$ be the subalgebra of $U_q(\lie{sl}_2)$ generated by $E^{(l)}$, $F^{(l)}$.
\begin{proposition}  \label{L(lm)} 
 Let $m = lm_1 \in \mathbb{Z}_{\geq 0}$, and consider the simple $U_q(\lie{sl}_2)$ module $L(m)$, and $x \in L(m)$ 
 be such that $\qbin{K ; 0}{l} x = \qbin{l m_1}{l} x$ and $E^{(l)}x =0$.  Then 
 \begin{enumerate} 
 \item $E, F, K - (-1)^{m_1}$ act as $0$ in $L(m)$. 
 \item $L(m) = \mathcal{U} \cdot x$. 
 \end{enumerate} 
\end{proposition} 
\begin{proof} 
Define inductively the subspaces in $L(m)$ by setting  $V_0 = \CC \cdot x$, $V_k = F^{(l)} V_{k-1} + E^{(l)} V_{k-1}$. 
We will prove by induction that $EV_k = FV_k=0$. 
 
By Theorem \ref{Weyl} we have 
\[ 0 \longrightarrow L(lm_1 -2) \longrightarrow W(l m_1)  \longrightarrow L(lm_1) \longrightarrow 0 . \]
Taking into account the definition of $W(l m_1)$, we have that $E \tilde{x} =0$,  $F \tilde{x} \in L(lm_1 -2)$, where $\tilde{x}$ is a lift of $x \in L(lm_1)$ to $W(lm_1)$. Therefore, we have $Ex = Fx =0$, and $E V_0 = F V_0 = 0$.  

Now consider $V_k = F^{(l)} V_{k-1} + E^{(l)} V_{k-1}$. 
In $U_q(\lie{sl}_2)$ we have (\cite{Lus89}, 4.1): 
\[  F E^{(l)}  = E^{(l)} F - \frac{-Kq + K^{-1} q^{-1}}{q-q^{-1}} E^{(l-1)} .\] 
\[  E F^{(l)}  = F^{(l)}E + F^{(l-1)} \frac{-Kq + K^{-1} q^{-1}}{q-q^{-1}} .\] 
Let $D$ denote $E$ or $F$. Then 
\[ D V_k \subset E^{(l)} D V_{k-1} + F^{(l)} D V_{k-1} + U_q E^{l-1} V_{k-1} + U_q F^{l-1} V_{k-1} .\] 
By the induction hypothesis, we have $DV_{k-1} =0, E^{l-1}V_{k-1} = 0, F^{l-1} V_{k-1} =0$, and this completes the induction step. We have $K y = q^{lm_1 \pm 2lk} y = (-1)^{m_1} y$ for any $y \in V_k$. Hence $E, F, K- (-1)^{m_1}$ act as zero on 
$V =\sum_{k} V_k$.  

By construction $E^{(l)} V \subset V$, $F^{(l)} V \subset V$, hence $V$ is a $U_q(\lie{sl}_2)$ submodule of $L(m)$, and since $L(m)$ is simple, we have $V = L(m)$.
\end{proof} 

Let $m = m_1 l + m_0$ such that $0 \leq m_0 \leq l-1$ and $m_1  \in \mathbb{Z}_{\geq 0}$.  Using the Hopf algebra structure of $U_q(\lie{sl}_2)$, we can consider the tensor product 
of simple $U_q(\lie{sl}_2)$ modules $L(m_0) \otimes L(lm_1)$.  By the Proposition \ref{L(lm)} above, we have 
\[ E^{(l)} (x_0 \otimes x_1) = E^{(l)} x_0 \otimes x_1  + K^l x_0 \otimes E^{(l)} x_1 , \] 
\[ F^{(l)} (x_0 \otimes x_1) = F^{(l)} x_0 \otimes K^{-l} x_1  + x_0 \otimes F^{(l)} x_1 . \] 

\begin{proposition} \label{tensor_pr} 
The following $U_q(\lie{sl}_2)$  modules are isomorphic: 
\[ L(lm_1 + m_0)  \cong L(m_0)  \otimes L(lm_1) . \] 
\end{proposition} 
\begin{proof} 
Let $z_0 \in L(m_0)$ be such that $E z_0 =0$, $K z_0 = q^{m_0} z_0$, and $z_1 \in L(lm_1)$ such that $E^{(l)} z_1 =0$ 
and $\qbin{K ; 0}{l} z_1 = \qbin{lm_1}{l}$. These vectors are unique up to a scalar, as they are obtained as the image of 
$x \in W(k)$ under the projection $W(k) \to L(k)$ for  $k=m_0$ and $k = lm_1$ respectively. 
Then 
\[ E(z_0 \otimes z_1) = E (z_0) \otimes z_1 = 0 , \quad  E^{(l)} (z_0 \otimes z_1) = 
K^l z_0 \otimes E^{(l)}  z_1 =0 ,\] 
so that $(z_0 \otimes z_1)$ is a primitive vector in $L(m_0) \otimes L(lm_1)$. 
Then applying $E^{(l)} F^{(l)} - F^{(l)} E^{(l)} = \sum\limits_{0 \leq t \leq l} F^{(l-t)} \qbin{K ; 2t-2l}{t} E^{(l-t)}$ (\cite{Lus89}, 4.1),  to the primitive vector $(z_0 \otimes z_1)$, we obtain 
\[ \qbin{K ; 0 }{l} (z_0 \otimes z_1) = K^l (z_0) \otimes (E^{(l)} F^{(l)} - F^{(l)} E^{(l)}) z_1 = \] 
\[ = (-1)^{m_0} z_0 \otimes \qbin{K ; 0}{l} (z_1) = (-1)^{m_0} \qbin{lm_1}{l} (z_0 \otimes z_1) = \qbin{lm_1 + m_0}{l} (z_0 \otimes z_1) .\] 
In the last equality we used the identity \ref{q-bin3}
We also have $K(z_0 \otimes z_1) =  q^{lm_1+m_0} (z_0 \otimes z_1)$. This implies that the vector $(z_0 \otimes z_1)$ has the same properties as a primitive vector of the simple module $L(lm_1 + m_0)$. 

Let us prove that $L(m_0) \otimes L(lm_1)$ is generated by $(z_0 \otimes z_1)$. Let $x_0 \otimes x_1 \in L(m_0) \otimes L(m_1)$, Since $z_0$ and $z_1$ generate $L(m_0)$ and $L(lm_1)$ over $u_q$ and $\mathcal {U}$ respectively, there exist elements $f \in u_q$ and $\mathcal{F} \in \mathcal{U}$ such that $x_0 = f z_0$ and $x_1 = \mathcal{F} z_1$. Then we have using Propositions \ref{Lm0}, \ref{L(lm)} 
\[ f \mathcal{F} (z_0 \otimes z_1) = f(\pm z_0 \otimes \mathcal{F} z_1) = \pm f(z_0 \otimes x_1) = \pm f(z_0) \otimes x_1 = \pm x_0 \otimes x_1 .\] 
So $L(m_0) \otimes L(lm_1)$ is  generated by $z_0 \otimes z_1$ of weight $lm_1+ m_0$. 

Suppose that there exists a nontrivial proper submodule $M \subset L(m_0) \otimes L(lm_1)$, and let $w \in M$ be primitive.  Let $\{v_j\}_{j=1}^r$ be a basis over $\CC$ of $L(lm_1)$. Then we can write uniquely 
$w = \sum_{j=1}^r w_j \otimes v_j$, where $w_j \in L(m_0)$. Since the vector $w$ is primitive in $M$, we have 
$Ew = \sum_{j=1}^r E w_j \otimes v_j =0$, hence $Ew_j =0$ for all $j$, therefore $w_j$ are multiples of $z_0 \in L(m_0)$, and $w = z_0 \otimes v$ for some $v \in L(lm_1)$. We also have 
$E^{(l)} w = z_0 \otimes E^{(l)} v =0$, and therefore $v$ is a multiple of $z_1 \subset L(lm_1)$. Therefore, any primitive vector in $M$ is a multiple of $z_0 \otimes z_1$, and this implies that $L(m_0) \otimes L(lm_1)$ is simple, and that the primitive vector in it has the weight $lm_1 + m_0$. Any such module is a simple quotient of $W(lm_1 + m_0)$ and therefore it is isomorphic to $L(lm_1+ m_0)$.
\end{proof}  

\begin{theorem} \label{tilt_prod}
For any $0 \leq k_0 \leq l-1$ and $k_1 \in \mathbb{Z}_{\geq 0}$ we have the tensor product identity for the indecomposable 
tilting modules over $U_q({\lie sl}_2)$: 
\[ T(k_1 l + 2(l-1) - k_0) \cong L(lk_1)  \otimes T(2(l-1) -k_0) .\] 
\end{theorem} 
\begin{proof} 
Using the filtrations of the tilting and Weyl modules in Theorems \ref{Weyl} \ref{tilt_prop2}, we conclude that $L(lk_1 +k_0)$ is the socle of the tilting module $T(lk_1 + 2(l-1) -k_0)$, and $L(lk_1) \otimes L(k_0)$ is the socle of the module 
$L(lk_1) \otimes T(2(l-1) - k_0)$. By the tensor product theorem \ref{tensor_pr}, these two simple modules are isomorphic. 
By a standard argument following \cite{APW1}, Section 9, \cite{APW2}, 4.6,  the tilting modules in both sides are injective; in particular, the module $T(l k_1+ 2(l-1) -k_0)$ is an injective hull of the simple module $L(lk_1 + k_0)$, and therefore it has to be a direct summand of the module  $L(lk_1) \otimes T(2(l-1) - k_0)$. Comparing the dimensions gives the isomorphism (we use that the dimensions of the Weyl modules in the filtration of tilting modules are given by the Weyl dimension formula). 
\end{proof}

\section{Decomposition of tensor powers of $T(1)$ over $U_q(\lie{sl}_2)$} \label{tp}
We are interested in the tensor powers of the module $L(1)$ over the quantum group $U_q(\lie{sl}_2)$, given explicitly by the following action of the generators: 
\[ F v_1  = v_{-1} \quad \quad F v_{-1} =0 , \quad \quad Ev_1= 0, \quad \quad Ev_{-1} =   v_1 .\] 
\[ Kv_1 =  q v_1, \quad \quad  Kv_{-1} = q^{-1} v_{-1} .\] 
\[ E^{(l)} v_i =0, \quad \quad F^{(l)} v_i =0, \quad \quad i=1, -1 .\] 
Since $L(1) \simeq T(1)$ is a tilting module, its tensor powers decompose into a direct sum of indecomposable tilting modules over $U_q(\lie{sl}_2)$.

Let $V(n)$ be the simple module of highest weight $n \in \mathbb{Z}_{\geq 0}$ over the Lie algebra $\lie{sl}(2, \CC)$.  Define the coefficients $F_k^{(N)}$ as multiplicities in the following decomposition: 
\begin{equation} \label{sl2} 
 V(1)^{\otimes N} = \bigoplus_{\substack{0 \leq k \leq N\\ k \equiv N\,{\rm mod}\,{2}}}  \; V(k)^{\oplus F_k^{(N)}}  \; . 
 \end{equation} 
i.e. $F_k^{(N)}=\dim\mathrm{Hom}(V(1)^{\otimes N}, V(k))$. It is known that  $F_k^{(N)}$ are 
\begin{equation} \label{tkN} 
F_k^{(N)} = {N\choose{\frac{N+k}{2}}} \frac{2k+2}{N+k+2}= {N\choose{\frac{N-k}{2}}}-{N\choose{\frac{N-k-2}{2}}} .
\end{equation}

The following theorem provides the mutliplicities of $T(m)$ in $(L(1))^{\otimes N}$ over the algebra $U_q(\lie{sl}_2)$. 
\begin{theorem} see also \cite{And17}, Corollary 3.5. \label{comb_coef}
 Using the classical coefficients $F_k^{(N)}$ defined in (\ref{tkN}),  we have  
\[ T(1)^{\otimes N} =  \bigoplus_{\substack{0 \leq k \leq N\\ k \equiv N\,{\rm mod}\,{2}}}  \; T(k)^{\oplus M_{k}^{(l)}(N)}  \; , \] 
where
\begin{enumerate} 
\item If $k = k_1 l -1 \leq N$ for $k_1 \geq 1$,  then  $M_{k_1 l -1}^{(l)}(N) = F_{k_1l-1}^{(N)}$,  
\item  If  $k = k_1 l + k_0 \leq N$ with $k_1 \geq 0$ and $0 \leq k_0 \leq l-2$, then 
\begin{eqnarray}\label{p-numbers}
M_{k_1 l +k_0}^{(l)}(N) =   
 \sum\limits_{m=0}^{\lfloor\frac{N-k}{2l}\rfloor} F_{(k_1+2m)l + k_0}^{(N)}-  \sum\limits_{n=0}^{\lfloor\frac{N-(k_1+ 2)l +k_0+2}{2l}\rfloor} F_{(k_1+2n +2)l -k_0-2}^{(N)} .  \nonumber
\end{eqnarray}
Here $\lfloor a\rfloor$ denotes the integer part (floor) of $a$. If the upper limit for $n$ is negative, the sum is void.   
\end{enumerate} 
\end{theorem} 
\begin{proof} 
By Theorem \ref{tilt_prop1}, the tensor power of $T(1)  = L(1)$ is a tilting module whose decomposition into the indecomposable tilting modules is determined by its character. Denote by $M_{k}^{(l)}(N)$ the multiplicity of the tilting module $T(k)$ in the decomposition of the tensor power $T(1)^{\otimes N}$ over $U_q(\lie{sl}_2)$. 

Note that the formula \ref{sl2} provides the coefficients $F_k^{(N)}$ of the decomposition of the character of $T(1)^{\otimes N}$ in terms of the characters of the Weyl modules over $U_q(\lie{sl}_2)$. Therefore our task reduces to rewriting this decomposition in the basis of the tilting characters. 

The decomposition of the basis of tilting characters in terms of the Weyl characters is given by Corollary  \ref{tilt_char}. Consider the vector space $Q_N$ spanned by the characters $\{ {\rm ch}_{W(k)} \}_{0 \leq k \leq N}$.  
Then the subspace $Q_{N, l-1} \subset Q_N$ spanned by $\{{\rm ch}_{W(k_1l-1)}\}_{0 < k_1l-1 \leq N}$ coincides with the subspace spanned by $\{{\rm ch}_{T(k_1l-1)}\}_{0 < k_1l-1 \leq N}$, 
and we have ${\rm ch}_{W(k_1l -1)} = {\rm ch}_{T(k_1l -1)}$, so that the change of basis matrix on $Q_{N, l-1}$ is trivial. This proves (a). 
 
The subspace $Q_{\rm tilt}$  of $Q_N= Q_{\rm tilt} \oplus Q_{N, l-1}$ spanned by the characters $\{{\rm ch}_{W(k_1l + k_0)} \}_{\substack{0 \leq k_0 \leq l-2,\\ k_1l + k_0 \leq N}}$ splits into a direct sum of subspaces 
\[ Q_{\rm tilt} =  \bigoplus_{0\leq k_0 \leq l-2} \; Q_{N, k_0} \] 
spanned by the Weyl characters of dominant highest weights in the orbit of $k_0$ with respect to the reflections about the weights of the form $\{ jl-1\}_{ j \in \Z}$. This follows from the formula for the tilting character ${\rm ch}_{T(lk_1 + k_0)}$, $0 \leq k_0 \leq l-2$ in Corollary \ref{tilt_char} and is a manifestation of the {\it linkage principle} (cf. \cite{AT}, 2.4). Let us fix $k_0$ such that  $0  \leq k_0 \leq l-2$. Then the weights of the Weyl (or tilting) characters spanning $Q_{N, k_0}$ are the following: 
\[  \begin{array}{ll} 
 \{ \lambda_{k_1} =   k_1 l + k_0 \}_{0 \leq k_1l + k_0 \leq N},   &    k_1   \; {\rm even}  \\
 \{ \lambda_{k_1} =  k_1 l + l-j_0 -2\}_{ 0 \leq  k_1 l + l - k_0 -2 \leq N}    &  k_1  \; {\rm odd} 
   \end{array}  \] 
   Let $k_{\rm max}$ denote the maximal value of $k_1$ satisfying the above conditions. Then  the Weyl characters that span $Q_{N, k_0}$ are enumerated by all integers $k_1$ such that $0 \leq k_1 \leq k_{\rm max}$, namely they are  
    $\{{\rm ch}_{W(\lambda_{k_1})} \}_{0 \leq k_1 \leq k_{\rm max}}$. 
The same space is also spanned by the characters  $\{ {\rm ch}_{T(\lambda_{k_1})} \}_{0 \leq k_1 \leq k_{\rm max}}$, and we can rewrite Corollary \ref{tilt_char} as follows:  
 \[\begin{array}{l} 
 {\rm ch}_{T(\lambda_0)} = {\rm ch}_{W(\lambda_0)} , \\ 
  {\rm ch}_{T(\lambda_{k_1})} = {\rm ch}_{W(\lambda_{k_1})} + {\rm ch}_{W(\lambda_{k_1-1})}   , \;\;\;\;  1 \leq k_1 \leq k_{\rm max}  . 
  \end{array}  \] 
  This defines the change of basis matrix between the tilting and Weyl characters and its inverse: 
  \[ \begin{array}{lll} 
  A = \left( \begin{array}{ccccc} 
              1 & 1 & 0 & 0 & 0 \\ 
                 & 1 & 1 & 0 & 0 \\ 
               \ldots  &   \ldots  &  \ldots & \ldots   & \ldots    \\
                 &    &    &  1  & 1  \\
                  &    &   &     &  1 \\    
                  \end{array} \right)  ,    &  \;\;\;\;\;   &
     A^{-1} = \left( \begin{array}{ccccc} 
              1 & -1 & 1 & -1 & 1 \\ 
                 & 1 & -1 & 1 & -1 \\ 
          \ldots  &   \ldots  &  \ldots & \ldots   & \ldots    \\
                 &    &    &  1  & -1  \\
                  &    &   &     &  1 \\    
                  \end{array} \right)     
                  \end{array} \] 
Applying the matrix $A^{-1}$ to the column of classical coefficients $\{ t_{k, N}\}$, we obtain the  formula given in part (b) of the theorem.                           
\end{proof}

\section{Restriction from  $U_q(\lie{sl}_2)$ to $u_q(\lie{sl}_2)$}\label{restUu}
Let $U_q^0 \subset U_q({\lie sl}_2)$ denote the subalgebra of $U_q({\lie sl}_2)$ generated by $K^{\pm 1}$, $\qbin{K ; c}{t}$ for  $t \geq 0, c \in \Z$. Let $u_q^{\pm} \subset u_q$ denote the subalgebras of $u_q$ generated by $F$ and $E$ respectively.  Using Theorem \ref{tilt_prod}, we can obtain the following result for the restriction of the tilting modules over the algebra 
$u_q^- U_q^0 u_q^+$: 
\begin{corollary}\label{restr} 
Restricting Theorem  \ref{comb_coef} over the algebra $u_q^- U_q^0 u_q^+$, we obtain after a change of variables 
\[  T(k_1 l + (l-1) + k_0) \big\vert_{u_q^- U_q^0 u_q^+} \cong \bigoplus_{\substack{-k_1 \leq j \leq k_1\\  j\equiv k_1 \,{\rm mod}\, 2}} \CC[lj]  \otimes  T((l-1) +k_0) .\] 
Here $k_1 \in \mathbb{Z}_{\geq 0}$, $0 \leq k_0 \leq l-1$, and $\CC[lj]$ is the one-dimensional module of weight $lj$. 
\end{corollary} 
Note that 
\[ K \CC[lj] = q^{lj} \CC[lj] = \left\{  \begin{array}{cl} 
                                          \CC[lj] ,   &    j \; {\rm even}  \\
                                          -\CC[lj],  & j \; {\rm odd }   
                                          \end{array}    \right.    
                                          \quad \quad  \quad   E \CC[lj] = F \CC[lj] = 0.  \] 

Here is the list of indecomposable modules that can occur in the tensor powers of $L(1) \cong  T(1)$ 
over $u_q^- U_q^0 u_q^+$: 
\[  \{ T(k_0)  \}_{0 \leq k_0 \leq l-2 } , \quad   \{  \CC[lj] \otimes T(l-1) \}_{j \in \Z} , \quad  
\{ \CC[lj] \otimes T(l + k_0) \}_{j \in \Z, \; 0 \leq k_0 \leq l-2} . \] 
Note that $T(k_0) \cong L(k_0)$, $ 0 \leq k_0 \leq l-2$ and $T(l-1)\cong L(l-1)$ are simple modules.

By Corollary \ref{restr}  we have the following decomposition over  $u_q^- U_q^0 u_q^+$: 
\[ 
T(1)^{\otimes N} \cong \bigoplus_{0 \leq k_0 \leq l-2, \; k_0 \leq N} T(k_0)^{\oplus M_{k_0}^{(l)}(N)} \oplus \bigoplus_{j \in \Z, \; l|j| + l-1 \leq N}  (\CC[lj] \otimes T(l-1))^{\oplus m^{(l)}_{l-1,j}(N)}  \oplus  
\] 
\[ \;\;\;\;\;\;\;\;\;\;\;\;\;\;\;\;\;\;\;\;\;\;\;\;\;\;\;\;\;\;\;\;\;\;\;\;\;\;\;\;\;\;\;\;\;\;\;\;\;\;\;\;\;\;\;\;\;\;\;  \bigoplus_{j \in \Z, \; 
0 \leq k_0 \leq l-2, \; l|j| + l + k_0 \leq N} (\CC[lj] \otimes T(l + k_0))^{\oplus m^{(l)}_{l+k_0,j}(N)} .\] 
Here we denoted  by $m^{(l)}_{l-1,j}(N)$ and $m^{(l)}_{l+k_0,j}(N)$ the multiplicities of $\CC[lj] \otimes T(l-1)$
and $\CC[lj] \otimes T(l + k_0)$ respectively.

\begin{theorem}  \label{uUu} 
Let $j \in \mathbb{Z}_{\geq 0}$, and $0 \leq k_0 \leq l-2$. The multiplicities $m^{(l)}_{l-1, \pm j}( N)$ and $m^{(l)}_{l+k_0, \pm j}( N)$ are given by the formulae 
\[ m^{(l)}_{l-1,\pm j}(N) =  \sum_{s \geq \lfloor \frac{j}{2} \rfloor}  M^{(l)}_{l-1+ j({\rm mod}2)l + 2sl}(N) ;  \] 
\[ m^{(l)}_{l+k_0,\pm j}(N)  = \sum_{s \geq \lfloor \frac{j}{2} \rfloor} M^{(l)}_{l+k_0+ j({\rm mod}2)l + 2sl}(N) . \]
The upper limit for $s$ can be calculated as $\lfloor \frac{N-(l-1+ j({\rm mod}2)l)}{2l}\rfloor$ or 
$\lfloor \frac{N-(l+k_0+ j({\rm mod}2)l)}{2l}\rfloor$, but in fact this is unnecessary because for higher $s$ 
 the coefficients $M^{(l)}_\bullet(N)$ under the sum are zero.  
\end{theorem} 
The proof of this is a direct consequence of Theorem \ref{comb_coef} and Corollary \ref{restr}.

Let $T(l + k_0)$, $0 \leq k_0 \leq l-2$ and $T(l-1)$ be the restrictions of the tilting modules with given highest weights to the small quantum group $u_q$. Following  \cite{APW2}, we conclude that these are indecomposable injective modules. Over the small quantum group $u_q(\lie{sl}_2)$, define the modules $T_{\pm}(m)$, $m \geq l-1$ as follows: 
\[ \CC[lj] \otimes T(l+ k_0) \cong T_+(l+ k_0), \quad   \CC[lj] \otimes T(l-1) \cong T_+(l-1)  \quad j  \; {\rm even}  ; \] 
\[ \CC[lj] \otimes T(l+ k_0) \cong T_-(l+ k_0), \quad   \CC[lj] \otimes T(l-1) \cong T_-(l-1)  \quad j  \; {\rm odd}  ; \] 
Notes that we have natural isomorphisms $T_-(k)\simeq T_+(k)\otimes 1_{-}$ where $1_-$ is a
1-dimensional representation $K\mapsto -1, E,F\mapsto 0$.

Then over the small quantum group we obtain the following decomposition. 
\begin{theorem}\label{small_mult}
\[ 
T(1)^{\otimes N} \cong \bigoplus_{0 \leq k_0 \leq l-2} T(k_0)^{\oplus M^{(l)}_{k_0}(N)} \oplus  T_+(l-1)^{\oplus m^{(l)}_{l-1,+}(N)}   \oplus  T_-(l-1)^{\oplus m^{(l)}_{l-1,-}(N)}  \oplus\]   
\[ \oplus  \bigoplus_{0 \leq k_0 \leq l-2}T_+(l + k_0)^{\oplus m^{(l)}_{l+k_0,+}(N)}    
 \oplus  \bigoplus_{0 \leq k_0 \leq l-2}T_-(l + k_0)^{\oplus m^{(l)}_{l+k_0,-}(N)} ,\] 
where 
\[ m^{(l)}_{l-1,+}(N) = \sum_{j \in \Z, \,{\rm even}} m^{(l)}_{l-1,j}(N) ,  \quad \quad  m^{(l)}_{l-1,-}(N) = \sum_{j \in \Z, \,{\rm odd}} m^{(l)}_{l-1,j}(N) ,   \] 
\[  m^{(l)}_{l+k_0,+}(N)  = \sum_{j \in \Z, \,{\rm even} } m^{(l)}_{l+k_0,j}(N)  \quad \quad  m^{(l)}_{l+k_0,-}(N) = \sum_{j \in \Z, \,{\rm odd} } m^{(l)}_{l+k_0,j}(N) . \] 
\end{theorem} 
This theorem follows directly from Theorem \ref{uUu} and the definition of the modules $T_\pm (l + k_0)$, $T_\pm(l-1)$.

\section{Recursive formulae for $U_q(\lie{sl}_2)$ and $u_q(\lie{sl}_2)$ tensor products.} \label{recursion}
Denote by $T(lk_1 + k_0)$ the indecomposable tilting module over $U_q ({\lie sl}_2)$ of type 1 with the highest weight $lk_1 + k_0$, with $0 \leq k_0 \leq l-1$, $k_1 \geq 0$. Note that modules $T(k)=L(k)=W(k)$ are irreducible when $0\leq k\leq l-1$ or $k=lk_1-1$.

Then we have:
\begin{eqnarray}\label{A}
T(k_0) \otimes T(1) &=& T(k_0+1) \oplus T(k_0 -1),  \quad 0 \leq  k_0 \leq l-2,  \\ 
T(lk_1 + k_0) \otimes T(1) &=& T(lk_1 + k_0 +1) \oplus T(lk_1 + k_0 -1) ,  \quad 1 \leq k_0 \leq l-3,\, k_1\geq 1 , \nonumber \\ 
T(lk_1 + l-2) \otimes T(1) &=& T(l(k_1 +1) -3) \oplus T((k_1+1)l-1) \oplus T((k_1-1)l-1) ,  \quad  k_1\geq 2, \nonumber \\ 
T(2l-2) \otimes T(1) &=& T(2l-1) \oplus T(2l-3) ,\nonumber \\
T(l k_1 -1) \otimes T(1) &=& T(lk_1) , \quad  k_1 \geq 1,\nonumber \\
T(lk_1) \otimes T(1) &=& T(lk_1 +1) \oplus 2 T(lk_1 -1) , \quad k_1\geq 1.\nonumber
\end{eqnarray}
We assume everywhere that $T(m) =0$ if $m <0$.

Similarly, we can derive recursive tensor product decomposition over the small quantum group $u_q$. 
Over $u_q(\lie{sl}_2)$ we have the following indecomposable modules that can appear in the powers of $T(1)=L(1)$ of type $1$:
$$
W(k_0)=L(k_0)=T(k_0),\quad\text{for $0 \leq k_0 \leq l-2$}, 
$$
$$
T_+ (l-1)=T(l-1),
$$
$$
T_+(l+ k_0)=T(l+k_0),\quad\text{for $0 \leq k_0 \leq l-2$},
$$
$$
T_-(l-1)=T(l-1)\otimes 1_-,
$$
$$
T_-(l+k_0)=T(l+k_0) \otimes 1_-,\quad\text{for $0 \leq k_0 \leq l-2$}. 
$$
Then we have:
\begin{eqnarray}\label{B}
T(k_0) \otimes T(1)  &=& T(k_0-1) \oplus T(k_0 +1) , \quad  1 \leq k_0 \leq l-2,  \\ 
T_{\pm}(l-1) \otimes T(1)  &=& T_{\pm}(l),\nonumber \\ 
T_{\pm}(l) \otimes T(1)  &=& T_{\pm}(l+1) \oplus 2 T_{\pm}(l-1),\nonumber \\ 
T_{\pm}(l+ k_0) \otimes T(1) &=& T(l +k_0 -1)_{\pm} \oplus  T_{\pm}(l+k_0 +1) ,\quad 1 \leq k_0 \leq l-3,\nonumber \\ 
T_{\pm}(2l-2) \otimes T(1) &=& T_{\pm}(2l-3) \oplus 2 T_{\mp} (l-1).\nonumber
\end{eqnarray}

\section{Multiplicities and lattice paths}\label{rec-mult}
Here we reformulate the problem of tensor power decomposition over $U_q(\mathfrak{sl}_2)$ and $u_q(\mathfrak{sl}_2)$ resolved in the previous section in terms of counting lattice paths on the set of dominant weights for the corresponding quantum groups and will give a combinatorial derivation of multiplicity formulae for $T(1)^{\otimes N}$ derived before.
\subsection{Lattice path model for tensor power decomposition of $U_q(\lie{sl}_2)$-modules}\label{lpm}
\subsubsection{Elementary steps and lattice paths.} Tensor product decomposition (\ref{A}) can be regarded as the description of a lattice walk on the set of integral dominant weights which is for $U_q(\lie{sl}_2)$ is $\Z_{\geq 0}=\{0,1,2,\dots\}$.

Let us describe elementary steps of such a walk. We will refer to points of the weight space $\Z_{\geq 0}$ also as nodes (of our lattice walk). We will say a step has multiplicity $2$ if it is counted twice (see below). If the multiplicity of a step is not specified, we assume it has multiplicity one.

\begin{enumerate}\label{def_elem_U}
\item From the node $k$, when $k\neq ml-2, ml-1, ml$ there are two possible steps, one to $k+1$ and one to $k-1$.
\item From the node $ml-2$ ($m\geq 3$) only steps to $ml-1, ml-3, (m-2)l-1$ are possible. From $2l-2$ there are only steps to $2l-1$ and $2l-3$.
\item From the node $ml-1$ there is one step leading to the node $ml$.
\item From the node $ml$ there is a multiplicity $2$ step to the node $ml-1$ and simple (multiplicity one) step to $ml+1$.
\end{enumerate}

Denote the set of these elementary steps $S_U$. They can be described by a diagram on Figure \ref{elem-steps}.
\begin{figure}[h!]
	\centerline{\includegraphics[width=350pt]
	{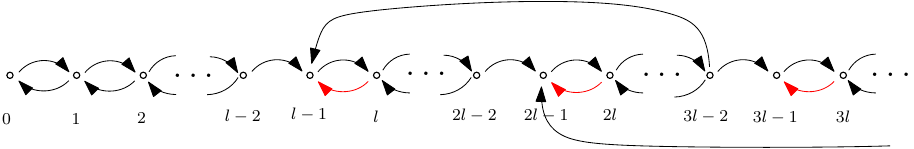}
}
\caption{The set of elementary steps $\Sb_U$, where red arrows denote steps of multiplicity 2.}
\label{elem-steps}
\end{figure}

A graph of such lattice walk is a lattice path in $\cL_U=\Z_{\geq 0}\times \Z_{\geq 0}$ with points $(x,y)$ where $x\in \Z_{\geq 0}$ is a dominant weight and $y$ is the degree $N$ is the tensor power $T(1)^{\otimes N}$.
The set of elementary steps $\Sb_U$ for such lattice paths is 
\[
\Sb_U=\{((x,y), (x',y+1))\,|\, (x,x')\in S_U\}
\]
Note that multiplicities of elementary steps define weights on $\Sb_U$. For large $k=lk_1+k_0$ the graph of elementary steps is $2l$-periodic.

Examples of such lattice paths are given on Figure \ref{two-paths}.
\begin{figure}[h!]
	\centerline{\includegraphics[width=250pt]
	{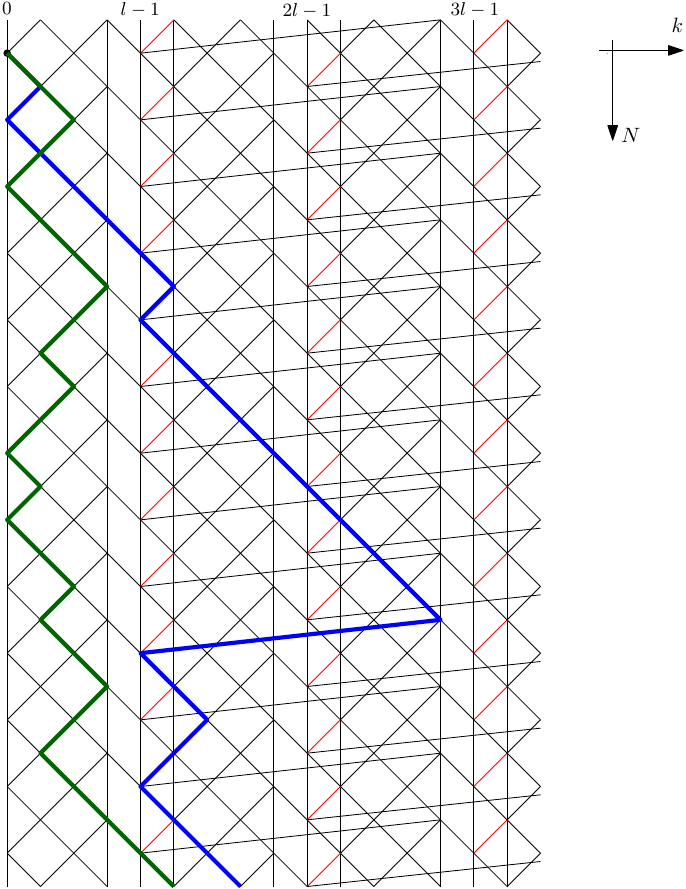}
}
	\caption{Two paths with elementary steps from $\Sb_U$. Red lines indicate elementary steps with multiplicity 2. A vertex with coordinate $(k,N)$ represents the highest weight $k$ of the tilting module $T(k)\subset T(1)^{\otimes N}$. The number of lattice paths from the origin to $(k,N)$ is the multiplicity $M_k^{(l)}(N)$. Elementary steps $\Sb_U$ from a vertex on this lattice are represented by possible lines down the lattice. Similarly, the number of paths starting at $(n,0)$ and descending to $(k,N)$ is the multiplicity of the tilting module $T(k)\subset T(n)\otimes T(1)^{\otimes N}$.}
\label{two-paths}
\end{figure}
Using recursive formulae (\ref{A}) for the decomposition of tensor products and the corresponding lattice path model, these multiplicities can be described as follows:
\begin{theorem}\label{U-model} The multiplicity of the tilting $U_q(\lie{sl}_2)$-module $T(k)$ in the decomposition 
of $T(1)^{\otimes N}$ is equal to the weighted number of lattice paths on $\cL_U$ connecting $(0,0)$ and $(k,N)$ with weights given by multiplicities of elementary steps $\Sb_U$.
\end{theorem}
Path counting formula for this lattice path model was derived in \cite{Sol} where the number of paths was computed explicitly in terms $F_k^{(N)}$ using combinatorial methods such as reflection principle. This formula coincides with expressions (\ref{p-numbers}).

\subsection{Lattice path model for decomposition of tensor powers of $u_q(\lie{sl}_2)$-modules}\label{small_lpm}
\subsubsection{Elementary steps and lattice paths} In this case the lattice is $\cL_u=\{0,1,\dots, 3l-2\}\times \Z_{\geq 0}$. The following set of elementary steps is determined by the decomposition rules (\ref{B}).
\begin{enumerate}\label{def_elem_u}
	\item From the node $0$ there is only one step to $1$.
	\item From the node $k$, when $k\neq ml-2,ml-1,ml$ there are two possible steps, one to $k+1$ and one to $k-1$.
	\item From the node $2l-2$ there is a step to $2l-1$ with multiplicity $2$ and a step to $2l-3$.
	\item From the node $3l-2$ there is a step to $l-1$ with multiplicity $2$ and a step to $3l-3$.
	\item From the node $ml-1$ there is only step to $ml$.
	\item From the node $ml$ there is a step to $ml-1$ with multiplicity $2$ and a step to $ml+1$.
\end{enumerate}
The set $S_u$ can be characterized by the diagram on Figure \ref{A-u-steps}.
\begin{figure}[h!]
	\centerline{\includegraphics[width=250pt]
	{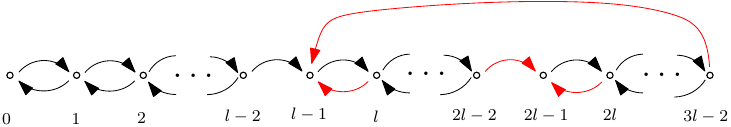}
}
	\caption{The set of elementary steps $\Sb_u$, where red arrows denote steps of multiplicity 2.}
\label{A-u-steps}
\end{figure}
A graph of such lattice walk is a lattice path in $\cL_u$ with points $(x,y)$ where $x\in \Z_{\geq 0}$ is a dominant weight and $y$ is the degree $N$ is the tensor power $T(1)^{\otimes N}$. Nodes of this graph are identified with tilting modules as follows:
\begin{itemize}
\item To each node $0\leq m \leq l-2$ we assign module $T(m)$, where $T(0)=\CC$;
\item To each node $l-1\leq m \leq 2l-2$ we assign module $T_+(m)$;
\item To each node $2l-1\leq m \leq 3l-2$ we assign module $T_-(m-l)$.
\end{itemize}
The set of elementary steps $\Sb_u$ is given by
\[
\Sb_u=\{((x,y), (x',y+1))\,|\, (x,x')\in S_u\}
\]
Multiplicities of elementary step define weights on $\Sb_u$.

Examples of such lattice paths are given on Figure \ref{two-paths-u}.
\begin{figure}[h!]
	\centerline{\includegraphics[width=250pt]
	{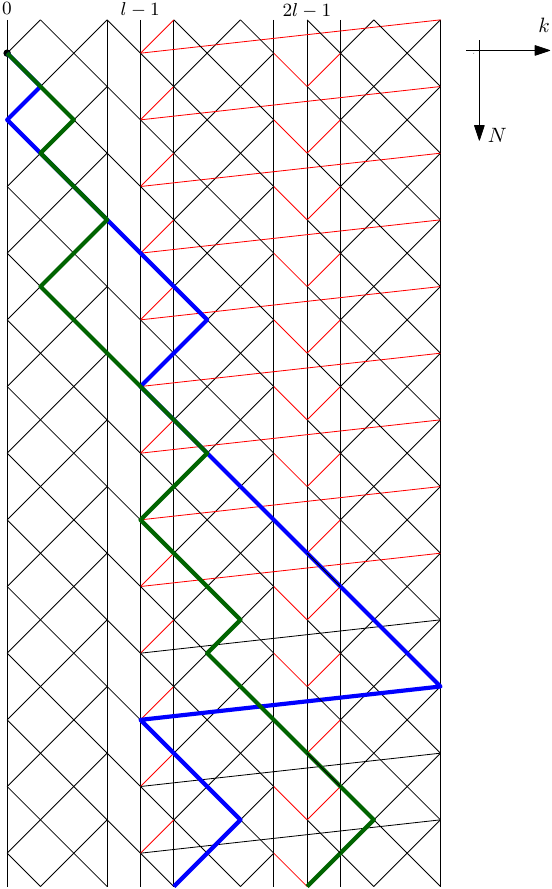}
}
	\caption{Two paths with elementary steps from $\Sb_u$. Red arrows indicate elementary steps with multiplicity 2. A vertex with coordinate $(k,N)$ represents the highest weight $k$ of the tilting module $T(k)\subset T(1)^{\otimes N}$. The number of lattice paths from the origin to $(k,N)$ is the multiplicity $m_k^{(l)}(N)$. Elementary steps $\Sb_u$ from a vertex on this lattice are represented by possible lines down the lattice. Similarly, the number of paths starting at $(n,0)$ and descending to $(k,N)$ is the multiplicity of the tilting module $T(k)\subset T(n)\otimes T(1)^{\otimes N}$.}
\label{two-paths-u}
\end{figure}
Using recursive formulae (\ref{B}) we arrive to the following theorem describing multiplicities in terms of lattice paths:
\begin{theorem} \label{u-model} The multiplicity of the tilting $u_q(\lie{sl}_2)$-module $T(k)$ in the decomposition of $T(1)^{\otimes N}$ is equal to the weighted number of lattice paths in $\cL_u$ connecting $(0,0)$ and $(k,N)$ with weights given by multiplicities of elementary steps $\Sb_u$.
\end{theorem} 
Now we can compute multiplicities using the lattice path model and compare the formulae with the ones that we already derived in Theorem \ref{small_mult}
\begin{theorem}\label{th_u} 
\begin{enumerate}
\item 
		The number of lattice paths is given by the following formulae.
		For $k_0\neq l-1$ we have
		\begin{equation}\label{ind1}
		m^{(l)}_{l+k_0,+}(N)=\sum_{k=1}^{\lfloor\frac{N}{2l}\rfloor}k^2F^{(N)}_{k_0-l-2kl}+\sum_{k=0}^{\lfloor\frac{N-l}{2l}\rfloor}(k+1)^2 F^{(N)}_{k_0+l+2kl}
		\end{equation}
		\begin{equation}\label{ind3}
		m^{(l)}_{l+k_0,-}(N)=\sum_{k=1}^{\lfloor\frac{N-l}{2l}\rfloor}(k^2+k)F^{(N)}_{k_0-2l-2kl}+\sum_{k=0}^{\lfloor\frac{N-2l}{2l}\rfloor}((k+1)^2+(k+1)) F^{(N)}_{k_0+2l+2kl}
		\end{equation}
		For $T_{+}(l-1)$, $T_{-}(l-1)$ the multiplicities are:
		\begin{equation*}
		m^{(l)}_{l-1,+}(N)=\sum_{k=0}^{\lfloor\frac{N}{2l}+\frac{1}{2}\rfloor}(2k+1)F^{(N)}_{l-1+2kl},
		\end{equation*}
		\begin{equation*}
		m^{(l)}_{l-1,-}(N)=\sum_{k=0}^{\lfloor\frac{N}{2l}+1\rfloor}(2k+2)F^{(N)}_{2l-1+2kl},
		\end{equation*}
		Note that in these formulae some $F^{(N)}_m$ can be negative.
\item These numbers are exactly the same as in Theorem \ref{small_mult}.
\end{enumerate}
\end{theorem}
The proof of this theorem is given in Appendix \ref{prthu}.

\section{Integral formula for multiplicities}\label{rec-int}
\subsection{Characters of tilting and Weyl modules.}
We define the character of an $U_q(\lie{sl}_2)$-module $V$ as the Laurent polynomial
(see section \ref{modules}) as
\[
\mathrm{ch}_V (x)=\sum_j \dim V_j x^j
\]
Here the sum is taken over weight subspaces of $V$ and $i$ ranges over all possible weights that occur in $V$. We will treat $x$ either as a formal variable, or, as a variable $x=e^{i\theta}$ on a unit circle in the complex plane. Because our characters are Laurent polynomials both make sense.

Theorem \ref{Weyl} and Theorem \ref{isomLW} give the following identifications and relations:
\begin{eqnarray*}
\mathrm{ch}_{L(k_0)}(x)=\mathrm{ch}_{W(k_0)}(x)=\mathrm{ch}_{T(k_0)}(x),& 0\leq k_0\leq l-1,\\
\mathrm{ch}_{W(k_1l+k_0)}(x)=\mathrm{ch}_{L(k_1l+k_0)}(x)+\mathrm{ch}_{L(k_1l-2-k_0)}(x),& 0\leq k_0\leq l-2,\,k_1\in\mathbb{Z}_{>0},\\
\mathrm{ch}_{T(k_1l+k_0)}(x)=\mathrm{ch}_{W(k_1l+k_0)}(x)+\mathrm{ch}_{W(k_1l-2-k_0)}(x),& 0\leq k_0\leq l-2,\,k_1\in\mathbb{Z}_{>0},\\
\mathrm{ch}_{L(k_1l+l-1)}(x)=\mathrm{ch}_{W(k_1l+l-1)}(x)=\mathrm{ch}_{T(k_1l+l-1)}(x),& k_1\in\mathbb{Z}_{>0}.
\end{eqnarray*}

\begin{lemma} \label{fchar}Characters of tilting modules are:
\begin{equation*}
\mathrm{ch}_{T(k_1l+k_0)}(x)=
\begin{cases}
\frac{x^{k_0+1}-x^{-(k_0+1)}}{x-x^{-1}},& \textit{if $0\leq k_0\leq l-1$ and $k_1=0$},\\
\frac{x^{(k_1+1)l}-x^{-(k_1+1)l}}{x-x^{-1}},& \textit{if $k_0=l-1$ and $k_1\in\mathbb{Z}_{>0}$},\\
(x^{k_0+1}+x^{-(k_0+1)})\frac{x^{k_1l}-x^{-k_1l}}{x-x^{-1}},& \textit{if $0\leq k_0\leq l-2$ and $k_1\in\mathbb{Z}_{>0}$}.
\end{cases}
\end{equation*}
	\end{lemma}
\begin{proof}
	Weight $k$ can be uniquely written as $k=k_1l+k_0$, where $0\leq k_0\leq l-1$ and $k_1\in\mathbb{Z}_{\geq 0}$. Characters for Weyl modules $W(k)$ are 	
	\begin{equation*}
	\mathrm{ch}_{W(k)}(x)=\frac{x^{k+1}-x^{-(k+1)}}{x-x^{-1}}.
	\end{equation*}
	From Theorem \ref{tensor_pr} we know, that 
	\begin{equation}\label{ch1}
	\mathrm{ch}_{L(k_1l+k_0)}(x)=\mathrm{ch}_{L(k_1l)}(x)\mathrm{ch}_{L(k_0)}(x),
	\end{equation}
	and for $0\leq k\leq l-2$ or $k=ml-1$, $m\in\mathbb{Z}_{>0}$
	\begin{equation}\label{ch2}
	\mathrm{ch}_{L(k)}(x)=\mathrm{ch}_{W(k)}(x)=\mathrm{ch}_{T(k)}(x).
	\end{equation}
	formulae (\ref{ch1}) and (\ref{ch2}) give
	\begin{equation*}
	\mathrm{ch}_{W((m+1)l-1)}(x)=\mathrm{ch}_{L(ml)}(x)\mathrm{ch}_{L(l-1)}(x)=\mathrm{ch}_{L(ml)}(x)\mathrm{ch}_{W(l-1)}(x),
	\end{equation*}
	which gives
	\begin{equation*}
	\mathrm{ch}_{L(ml)}(x)=\frac{\mathrm{ch}_{W((m+1)l-1)}(x)}{\mathrm{ch}_{W(l-1)}(x)}=\frac{x^{(m+1)l}-x^{-(m+1)l}}{x^l-x^{-l}}
	\end{equation*}
	and, concluding, we have
	\begin{equation*}
	\mathrm{ch}_{L(k_1l+k_0)}(x)=\frac{x^{k_0+1}-x^{-(k_0+1)}}{x-x^{-1}}\frac{x^{(k_1+1)l}-x^{-(k_1+1)l}}{x^l-x^{-l}}.
	\end{equation*}
	Note that this definition satisfies condition, given by short exact sequences in Theorem \ref{Weyl}, meaning that for $k_0\neq l-1$ we indeed have
	\begin{equation*}
	\mathrm{ch}_{W(k_1l+k_0)}(x)=\mathrm{ch}_{L(k_1l+k_0)}(x)+\mathrm{ch}_{L(k_1l-k_0-2)}(x).
	\end{equation*}
	Using exact sequences from Theorem \ref{isomLW}, we can derive characters for the rest of the tilting modules. Suppose $k_0\neq l-1, \ \ k_1\in\mathbb{Z}_{>0}$, then
	\begin{eqnarray*}
	\mathrm{ch}_{T(k_1l+k_0)}(x)=\mathrm{ch}_{W(k_1l+k_0)}(x)+\mathrm{ch}_{W(k_1l-k_0-2)}(x)=\\
	=(x^{k_0+1}+x^{-(k_0+1)})\frac{x^{k_1l}-x^{-k_1l}}{x-x^{-1}}\nonumber,
	\end{eqnarray*}
	which concludes the proof.
\end{proof}

\subsection{Scalar product and orthogonality.}
Consider the scalar product on functions on the unit circle which is the radial part of the Haar measure on $SU(2)$. For two functions $f, g: S^1=\{z=e^{i\theta}\}\to \CC$ define
\begin{equation}\label{sp}
(f,g)=\frac{1}{4\pi}\int_0^{2\pi}\overline{f(e^{i\theta})}g(e^{i\theta})|e^{i\theta}-e^{-i\theta}|^2d\theta
\end{equation}
In particular it defines the scalar product on the space of Laurent polynomials in $x$ which are symmetric with respect to the Weyl reflection $x\mapsto x^{-1}$.

We will use the following notations for characters of irreducible representations $L(k)$, Weyl modules $W(k)$ and tilting modules $T(k)$.
\begin{equation*}
\mathrm{ch}_{L(k_1l+k_0)}(x)=\chi^L_{k_1,k_0},\quad\mathrm{ch}_{W(k_1l+k_0)}(x)=\chi^W_{k_1,k_0},\quad\mathrm{ch}_{T(k_1l+k_0)}(x)=\chi^T_{k_1,k_0}.
\end{equation*}

Characters of Weyl modules form an orthonormal basis in the space $L^2(S^{1})$ with respect to the scalar product (\ref{sp}):
\begin{eqnarray*}
(\chi^W_{k_1,k_0},\chi^W_{m_1,m_0})=\delta_{k_1,m_1}\delta_{m_0,k_0},
\end{eqnarray*}
Scalar products between characters of tilting modules and Weyl modules are:
\begin{equation*}
(\chi^L_{k_1,k_0},\chi^T_{m_1,m_0})=
\begin{cases}
\delta_{l-2-k_0,m_0}(\delta_{k_1+1,m_1}-\delta_{0,m_1}),& \textit{if $k_0\neq l-1$, $k_1\in\mathbb{Z}_{>0}$},\\
\delta_{k_1,m_1}\delta_{k_0,m_0},& \textit{otherwise}.
\end{cases}
\end{equation*}

In order to find an integral formula for multiplicities, similar to the one that exists for compact Lie groups and is based on the orthogonality of characters of irreducible modules, let us find the basis orthogonal to the basis of characters of tiltling modules. 
\begin{theorem} Functions 
\begin{equation*}
\eta_{k_1,k_0}=\frac{(x^{k_1l-1}+x^{-(k_1l-1)})(x^{k_0-l+1}-x^{-(k_0-l+1)})}{(x-x^{-1})(x^l-x^{-l})}.
\end{equation*}
form a basis orthogonal to the basis of characters of tilting modules:
\begin{eqnarray*}
	(\chi^T_{m_1,m_0},\eta_{k_1,k_0})=\delta_{m_1,k_1}\delta_{m_0,k_0}
\end{eqnarray*}
\end{theorem}
\begin{proof} 
It is easy to find the decomposition of $\eta_{k_1,k_0}$ in the basis of Weyl modules:
\begin{equation*}
\eta_{k_1,k_0}=\sum_{j=0}^\infty (\chi^W_{k_1+2j,k_0}-\chi^W_{k_1+1+2j,l-2-k_0}).
\end{equation*}
Now let us use formulae expressing characters of tilting modules in terms of character of Weyl modules, the formula above and the orthogonality of characters of Weyl modules.
For $k_0\neq l-1$ and $k_1\in\mathbb{Z}_{>0}$ we have
\begin{eqnarray*}
(\chi^T_{m_1,m_0},\eta_{k_1,k_0})=\sum\limits_{j=0}^\infty((\chi^W_{m_1,m_0},\chi^W_{k_1+2j,k_0})+(\chi^W_{m_1-1,l-2-m_0},\chi^W_{k_1+2j,k_0})\\
-(\chi^W_{m_1,m_0},\chi^W_{k_1+1+2j,l-2-k_0})-(\chi^W_{m_1-1,l-2-m_0},\chi^W_{k_1+1+2j,l-2-k_0}))=\\
=\sum\limits_{j=k_1}^\infty(\delta_{m_1,k_1+2j}\delta_{m_0,k_0}+\delta_{m_1-1,k_1+2j}\delta_{l-2-m_0,k_0}-\\
-\delta_{m_1,k_1+1+2j}\delta_{m_0,l-2-k_0}-\delta_{m_1-1,k_1+1+2j}\delta_{l-2-m_0,l-2-k_0}).
\end{eqnarray*}
It is evident that this expression gives the right combination of Kroneker delta symbols:
\begin{eqnarray*}
	(\chi^T_{m_1,m_0},\eta_{m_1,k_0})=\delta_{m_0,k_0},
\end{eqnarray*}
and for $k_1<m_1$, again, $0$. To resume, we have
\begin{equation*}
	(\chi^T_{m_1,m_0},\eta_{k_1,k_0})=\delta_{m_0,k_0}\delta_{m_1,k_1},
\end{equation*}
\end{proof}

\subsection{Multiplicities}\label{integralchar}
Now we can use the basis $\eta_{k_1,k_0}$ to derive the integral formula for multiplicities of tilting modules. 

The tensor product decomposition, implies the equality of characters:
\begin{equation*}
(\mathrm{ch}_{T(1)}(x))^N=\sum_{0\leq k_0\leq l-1,\, k_1\in\mathbb{Z}_{\geq 0}}M^{(l)}_{k_1l+k_0}(N)\mathrm{ch}_{T(k_1l+k_0)}(x).
\end{equation*}
Now that we have found the orthogonal basis, we can use it to derive  the following integral formula for the multiplicity
\begin{equation*}
M^{(l)}_{k}(N)=\frac{1}{4\pi}\int_{0}^{2\pi}\overline{\eta_{k_1,k_0}(e^{i\theta})}(\mathrm{ch}_{T(1)}(e^{i\theta}))^N|e^{i\theta}-e^{-i\theta}|^2d\theta.
\end{equation*}

To compare this formula with Theorem \ref{comb_coef}, we use the definition of $\eta_{k_1,k_0}$, and decomposition into a sign-alternating sum of characters of Weyl modules. Using the orthogonality of Weyl characters we arrive to the expression
\begin{equation*}
M^{(l)}_{k}(N)=\sum_{j=0}^\infty (F_{k_1l+2jl+k_0}^{(N)}-F_{k_1l+2(j+1)l-2-k_0}^{(N)}),
\end{equation*}
where $F^{(N)}_m$ are the same as in (\ref{tkN}). 
Clearly, this agrees with Theorem \ref{comb_coef}. Moreover the analytic derivation above is more or less a rephrasing of the algebraic proof of this theorem.

\section{Probabilistic measures on the tensor product decomposition}\label{probmeas}
Here we introduce some natural probabilistic measures on indecomposable components of $T(1)^{\otimes N}$ for big and small quantum $\lie{sl}_2$ and we show that they are all induced by Markov processes. 
\subsection{The character and the Plancherel measures for $U_q(\lie{sl}_2)$ }
Let $X^{(l)}_N$ be the set of tilting modules of $U_q(\lie{sl}_2)$ that occur in the decomposition of $T(1)^{\otimes N}$. It can be naturally identified with the subset of dominant integral weights of $\lie{sl}_2$, i.e.  with $k\in \Z_{\geq 0}$ such that $k\leq N$, $k\equiv N\,{\rm mod}\,2$. Define the multiplicity function on $\Z_{\geq 0}$ as the multiplicity of $T(k)\subset T(1)^{\otimes N}$ when $k\in X^{(l)}_N$ and zero otherwise.

\begin{definition} The character measure on  $\Z_{\geq 0}$ is given by
\begin{equation*}
p^{(l)}_N(k;t)=\frac{M^{(l)}_{k}(N)\mathrm{ch}_{T(k)}(e^t)}{(\mathrm{ch}_{T(1)}(e^t))^N}.
\end{equation*}
Here $M^{(l)}_{k}(N)$ is the multiplicity function, $t$ is a real parameter and $\mathrm{ch}_{T(k)}(e^t)$ is
the formal character from Lemma \ref{fchar} evaluated on $e^t$. 
\end{definition}

Multiplicativity and additivity of characters imply that the character measure a probability measure, i.e. $\sum\limits_{k\in X^{(l)}_N} p^{(l)}_N(k;t)=1$. Recall that the multiplicity function is computed explicitly.
\begin{definition}\label{planchdef}
When $t=0$ the character measure becomes the Plancherel measure 
\begin{equation*}
p^{(l)}_N(k)=\frac{M^{(l)}_{k}(N)\dim T(k)}{(\dim  T(1))^N}.
\end{equation*}
\end{definition}

Dimensions of tilting modules can be easily derived from their formal characters:
\begin{equation*}
\dim T(k_0+lk_1)=
\begin{cases}
k_0+1,&\text{if $0\leq k_0\leq l-2$ and $k_1=0$,}\\
l(k_1+1), & \text{if $k_0=l-1$ and $k_1\in\mathbb{Z}_{>0}$,}\\
2k_1l,&\text{if $0\leq k_0\leq l-2$ and $k_1\in\mathbb{Z}_{>0}$.}
\end{cases}
\end{equation*}
It is clear that the Plancherel measure for $U_q(\lie{sl}_2)$ is the specialization of the character measure at $t=0$. We distinguish it because as we will see that limit $N\to \infty$ is not interchangable with the limit $t\to 0$.

\subsection{The Plancherel measure for  $u_q(\lie{sl}_2)$}\label{secplanch}
Now let us define the Plancherel measure on indecomposable components 
of $T(1)^{\otimes N}$ for $u_q(\lie{sl}_2)$. Recall that in Section \ref{small_lpm} to each node of a graph in Figure \ref{A-u-steps} we assigned a tilting module. From now on we will use notation corresponding to this assignment:
\begin{itemize}
    \item $T(k)=T_+(k)$ when $l-1 \leq k\leq 2l-2$ and $m^{(l)}_{k}(N)=m^{(l)}_{k, +}(N)$,
    \item $T(k)=T_-(k-l)$ when $2l-1 \leq k\leq 3l-2$ and $m^{(l)}_{k}(N)=m^{(l)}_{k-l, -}(N)$.
\end{itemize}

In this case, dimensions of indecomposable representations that occur in the decomposition of the tensor product $T(1)^{\otimes N}$ are written as follows
\begin{equation*}
\dim T(k)=
\begin{cases}
k+1,&\text{if $0 \leq k\leq l-2$,}\\
l, & \text{if $k=l-1,2l-1$,}\\
2l,& \text{otherwise.}
\end{cases}
\end{equation*}

\begin{definition}The Plancherel measure on indecomposable components of $u_q$-module $T(1)^{\otimes N}$ is 
\begin{equation*}
p^{(l)}_N(k)=\frac{m^{(l)}_{k}(N)\dim T(k)}{(\dim  T(1))^N},\quad 0\leq k\leq 3l-2
\end{equation*}
and is zero for $k>3l-2$.
\end{definition}

\subsection{The quantum Plancherel measure for $u_q(\lie{sl}_2)$.}\label{qdims}
Quantum dimension \cite{KR} of the irreducible $U_v(\lie{sl}_2)$-module $V(k)$ is
\[
\dim_v(V(k))=\frac{v^{k+1}-v^{-k-1}}{v-v^{-1}}.
\]
For quantum dimensions of modules that $T(k)$ which occur in the decomposition of $T(1)^{\otimes N}$ we have
\[
\dim_q(T(k))=\frac{q^{k+1}-q^{-k-1}}{q-q^{-1}}
\]
for $0\leq k\leq l-2$. For other modules $T(k)$ quantum dimensions are zero.

Note that when $q=e^{\frac{i\pi}{l}}$ we have $\dim_q(T(k))\in\mathbb{R}_{\geq 0}$ for arbitrary values of $k$ and $\dim_q(T(k))\in\mathbb{R}_{> 0}$ for $0\leq k\leq l-2$.\footnote{This positivity is closely related to the unitarity of Wess-Zumino-Witten conformal field theory and
of quantum Chern-Simons topological quantum field theory.}. It is not
difficult to show that this is the only root of degree $l$ from $-1$ with such property.

\begin{definition}
Quantum Plancherel measure is defined as
\begin{equation*}
p^{(l)}_N(k)=\frac{m^{(l)}_{T(k)}(N)\mathrm{dim}_qT(k)}{(\mathrm{dim}_q T(1))^N},\quad 0\leq k\leq l-2
\end{equation*}
and is zero for $k>l-2$.
\end{definition}

\section{The asymptotic of the character and Plancherel measures for $U_q(\lie{sl}_2)$}\label{as9}
In this section we will compute the asymptotic of the character and the Plancherel measures on tilting submodules of $T(1)^{\otimes N}$ when $N\to \infty$. We will start with computing the asymptotic of the multiplicity of $T(k)\subset T(1)^{\otimes N}$.
\subsection{The asymptotic of the multiplicity function}
Note that in the tensor product $T(1)^{\otimes N}$ only submodules $T(k)$ with $k\equiv N \,\mathrm{mod}\,2$ can occur. For sufficiently large $k$ the decomposition of $T(1)\otimes T(k)$ is $2l$ periodic. This is why it is natural to write $N$ in base $2l$:
$$
N=2lN_1+N_0,\quad N_0=0,1,\ldots,2l-1.
$$
Writing $k=lk_1+k_0$ we see that $N=k\,\mathrm{mod}\,2$ have the following implication:
\begin{itemize}
    \item $N_0$ is even and $k\equiv 0\,\mathrm{mod}\,2$ or $N_0$ is odd and $k\equiv 1\,\mathrm{mod}\,2$ imply $k_0\equiv N_0\,\mathrm{mod}\,2$
    \item $N_0$ is even and $k\equiv 1\,\mathrm{mod}\,2$ or $N_0$ is odd and $k\equiv 0\,\mathrm{mod}\,2$ imply $k_0\equiv (N_0+1)\,\mathrm{mod}\,2$
\end{itemize}
Let us define a parameter
$$
\gamma=
\begin{cases}
1,&\text{if $k_0\equiv (N_0+1)\,\mathrm{mod}\,2$,}\\
0,&\text{if $k_0\equiv N_0\,\mathrm{mod}\,2$.}
\end{cases}
$$
In the limit $N\to\infty$ $k_0$ remains discrete, and its parity is controlled by $\gamma$. Later we will see that limiting distribution is independent of both $N_0$ and $\gamma$.

The following theorem describes the pointwise asymptotic of the multiplicity function
\begin{theorem}\label{multasymptmain}
	Assume that $N=2lN_1+N_0$, $k=k_1l+k_0$, $N_1\to\infty$, $k_1\to\infty$ such that $l$, $N_0=0, \ldots, 2l-1$, $k_0=0, \ldots, l-1$ and $\xi=\frac{k_1}{2N_1}$ remain fixed. Then the asymptotic of the multiplicity has the following form:
    \begin{equation}\label{mass}
        M^{(l)}_{T(k)}(N)=\frac{\mu_{k_0}^{(l)}(\xi,N_0,\gamma)}{\sqrt{2\pi\, 2lN_1}}\mathrm{exp}(2lN_1 S(\xi))\left(1+O\left(\frac{1}{N_1}\right)\right)
    \end{equation}
    where
	\begin{equation*}
		S(\xi)=-\left(\frac{1-\xi}{2}\right)\mathrm{ln}\left(\frac{1-\xi}{2}\right)-\left(\frac{1+\xi}{2}\right)\mathrm{ln}\left(\frac{1+\xi}{2}\right),
	\end{equation*}
    and if $k_0\neq l-1$
    \begin{equation*}
		\mu_{k_0}^{(l)}(\xi,N_0,\gamma)=\frac{2^{2+N_0}\xi}{\sqrt{(1-\xi^2)^{1+N_0}}}\Big(\frac{1-\xi}{1+\xi}\Big)^{\frac{\gamma l}{2}}\frac{\frac{1}{1+\xi}\Big(\frac{1+\xi}{1-\xi}\Big)^{\frac{l-k_0}{2}}-\frac{1}{1-\xi}\Big(\frac{1+\xi}{1-\xi}\Big)^{-\frac{l-k_0}{2}}}{\Big(\frac{1+\xi}{1-\xi}\Big)^{\frac{l}{2}}-\Big(\frac{1+\xi}{1-\xi}\Big)^{-\frac{l}{2}}},
    \end{equation*}
    if $k_0=l-1$
    \begin{equation*}
		\mu_{k_0}^{(l)}(\xi,N_0,\gamma)=\frac{2^{2+N_0}\xi}{\sqrt{(1-\xi^2)^{1+N_0}}}\frac{1}{1+\xi}\Big(\frac{1-\xi}{1+\xi}\Big)^{\frac{l(1+\gamma)-1}{2}}.
    \end{equation*}
\end{theorem}
The proof of this theorem is presented in Appendix \ref{apmultas}.

\subsection{The pointwise asymptotic of the character measure density}
Now let us find the asymptotic of the character measure density function $p^{(l)}_N(k;t)$ when $k$, $N\to\infty$ in the same way as in Theorem \ref{multasymptmain}. All we have to do is to combine the results of Theorem \ref{multasymptmain} with the asymptotical expansions of characters.

Using explicit formulae for characters we find that when $k_0\neq l-1$
\begin{eqnarray*}
	\frac{\mathrm{ch}_{T(k)}(e^t)}{(\mathrm{ch}_{T(1)}(e^t))^N}=\frac{(e^{t(k_0+1)}+e^{-t(k_0+1)})(e^{tk_1l}-e^{-tk_1l})}{e^{t}-e^{-t}}\Big(\frac{e^{2t}-e^{-2t}}{e^{t}-e^{-t}}\Big)^{-N}=\\
	=\frac{(e^{t(k_0+1)}+e^{-t(k_0+1)})e^{t \gamma l}}{(e^{t}-e^{-t})(e^t+e^{-t})^{N_0}}\Big(\frac{e^{\xi t}}{e^{t}+e^{-t}}\Big)^{2lN_1}(1+o(1)),\nonumber
\end{eqnarray*}
and when $k_0=l-1$
\begin{eqnarray*}
	\frac{\mathrm{ch}_{T(lk_1+l-1)}(e^t)}{(\mathrm{ch}_{T(1)}(e^t))^N}=\frac{e^{tl(1+\gamma)}}{(e^{t}-e^{-t})(e^{t}+e^{-t})^{N_0}}\Big(\frac{e^{\xi t}}{e^{t}+e^{-t}}\Big)^{2lN_1}(1+o(1)).
\end{eqnarray*}

Combining these formulae with the results of Theorem \ref{multasymptmain} we obtain that if $N=2N_1l+N_0$, $k=k_1l+k_0$, and $N_1$, $k_1\to\infty$ such that $l$, $k_0$, $N_0$, $t$ and $\xi=\frac{k_1}{2N_1}$ remain fixed we have the following pointwise asymptotic for the probability measure density function
\begin{equation}\label{charas}
    p^{(l)}_N(k;t)=\frac{\nu_{k_0}^{(l)}(\xi,N_0,\gamma;t)}{\sqrt{2\pi\, 2lN_1}}\mathrm{exp}(2lN_1 (S(\xi)-f(t)+t\xi))\left(1+O\left(\frac{1}{N_1}\right)\right),
\end{equation}
where
\begin{equation*}
	f(t)=\mathrm{ln}(e^t+e^{-t}),
\end{equation*}
and if $k_0\neq l-1$
\begin{eqnarray*}
    &\nu_{k_0}^{(l)}(\xi,N_0,\gamma;t)=\frac{e^{t\gamma l}(e^{t(k_0+1)}+e^{-t(k_0+1)})}{(e^{t}-e^{-t})(e^t+e^{-t})^{N_0}}\frac{2^{2+N_0}\xi}{\sqrt{ (1-\xi^2)^{1+N_0}}}\Big(\frac{1-\xi}{1+\xi}\Big)^{\frac{\gamma l}{2}}\frac{\frac{1}{1+\xi}\Big(\frac{1+\xi}{1-\xi}\Big)^{\frac{l-k_0}{2}}-\frac{1}{1-\xi}\Big(\frac{1+\xi}{1-\xi}\Big)^{-\frac{l-k_0}{2}}}{\Big(\frac{1+\xi}{1-\xi}\Big)^{\frac{l}{2}}-\Big(\frac{1+\xi}{1-\xi}\Big)^{-\frac{l}{2}}},
\end{eqnarray*}
if $k_0=l-1$
\begin{equation*}
    \nu_{k_0}^{(l)}(\xi,N_0,\gamma;t)=\frac{e^{tl(1+\gamma)}}{(e^{t}-e^{-t})(e^t+e^{-t})^{N_0}}\frac{2^{2+N_0}\xi}{\sqrt{(1-\xi^2)^{1+N_0}}} \frac{1}{1+\xi}\Big(\frac{1-\xi}{1+\xi}\Big)^{\frac{l(1+\gamma)-1}{2}}.
\end{equation*}

By definition $\xi\in[0,1)$. Note that the function $S(\xi)$, the large deviation rate function has its maximum at the isolated critical point $\xi_0$, i.e.
\begin{equation*}
\frac{\partial}{\partial\xi}(S(\xi)-f(t)+t\xi)=0
\end{equation*}
has a unique solution
\begin{equation}\label{xi0}
\xi_0=\frac{e^t-e^{-t}}{e^t+e^{-t}}.
\end{equation}
In addition, $S(\xi_0)=0$.

Thus, we have a typical large deviation behavior of the density function $p^{(l)}_N(k;t)$: it is exponentially decaying when $\xi\neq\xi_0$ and it decays as $N_1^{-\frac{1}{2}}$ when $\xi=\xi_0$. Thus, asymptotically, the character measure concentrates in the neighborhood of $\xi=\xi_0$.

\subsection{The asymptotic of the character measure for $U_q(\lie{sl}_2)$}\label{aschar}
The following theorem describes the weak limit of the character probability measure for $U_q(\lie{sl}_2)$.
\begin{theorem}\label{theomch}
    Let $\xi_0$ be as in (\ref{xi0}), $N=2lN_1+N_0$ and $N_1\to\infty$ such that $N_0$, $l$, $t$ are fixed. Then for $k=k_1l+k_0$, random variables $k_0$ and $\alpha=\frac{k_1-2N_1 \xi_0}{\sqrt{2N_1}}$ converge to random variables on $\{0,\ldots,l-1\}\times\mathbb{R}$ and the character probability distribution weakly converges to the tensor product of two distributions $p_1(\bullet;t)\otimes p_2(\bullet;t)$ on $\{0,\ldots,l-1\}\times\mathbb{R}$. The distribution $p_1(\bullet;t)$ is a discrete probability distribution on $k_0=0,1,\ldots,l-1$ with density given by
	\begin{equation}\label{p1}
		p_1(k_0;t)=\begin{cases}
			\frac{(e^{t(k_0+1)}+e^{-t(k_0+1)})(e^{t(l-1-k_0)}-e^{-t(l-1-k_0)})}{l(e^{tl}-e^{-tl})},& \textit{if $k_0=0,\ldots,l-2$},\\
			\frac{1}{l},& \textit{if $k_0=l-1$}.
		\end{cases}
	\end{equation}
	The distribution $p_2(\bullet;t)$ is a Gaussian distribution on $\mathbb{R}$ with density given by
	\begin{equation}\label{p2}
		p_2(\alpha;t)=\sqrt{\frac{l}{2\pi}}\mathrm{cosh}(t)e^{-\frac{\alpha^2l}{2}\left(\mathrm{cosh}(t)\right)^2}.
	\end{equation}
\end{theorem}
The proof of this theorem is presented in Appendix \ref{approofchar}. The meaning of weak convergence is also recalled there.

Note that the asymptotic of the character measure is derived for $N_0$ being fixed, but it does not depend on $N_0$. The only dependence is the condition $N=k\,\mathrm{mod}\,2$. It implies that:
\begin{itemize}
    \item for even $N_0$ the summation over $k=k_1 l+k_0$ amounts for the summation over $k_0$ and $k_1$ of the same parity,
    \item for odd $N_0$ the summation over $k=k_1 l+k_0$ amounts for the summation over $k_0$ and $k_1$ of the opposite parity.
\end{itemize}
In the limit $N_1\to\infty$ this dependence on $N_0$ disappears. Dependence on $\gamma$ disappears as well. Period $2l$ covers $k_1$, $k_1+1$, which for any $k_1$ includes all values of $k_0=0,\ldots,l-1$, so parity of $k_0$ doesn't have to be controlled.

This is why we conjecture that the asymptotic does not depend on $N_0$. To be precise, consider a sequence
$$
N^{(p)}=2lN_1^{(p)}+N_0^{(p)},\quad N_1^{(p)}\to\infty\text{ as }p\to\infty,
$$
and $N_0^{(p)}=0,1,\ldots,2l-1$. For each $p\in\mathbb{Z}_{>0}$ let $k=lk_1+k_0\in\mathbb{Z}_{\geq 0}$ be the random variable distributed with respect to the character measure $p_{N^{(p)}}^{(l)}(k;t)$. We expect the following is true. As $p\to\infty$ with $l$, $t$ being fixed and $k=2N_1^{(p)}\xi_0+\sqrt{2N_1^{(p)}}\alpha$, $\alpha\in\mathbb{R}$, the character measure converges to the measure on $\{0,\ldots,l-1\}\times\mathbb{R}$ with the density $p_1(k_0;t)p_2(\alpha;t)$ as above, and the limit does not depend on the choice of sequences $\{N_1^{(p)}\}_{p\in\mathbb{Z}_{\geq 0}}$, $\{N_0^{(p)}\}_{p\in\mathbb{Z}_{\geq 0}}$.

\subsection{The asymptotic of the Placherel measure}
Now let us find the asymptotic of the Plancherel measure when $N\to\infty$ and $l$ is fixed. From definition \ref{planchdef} it is evident that the same ideas as explored in Subsection \ref{aschar} are applicable to the Plancherel measure.

In the considered limit we have the following asymptotical expression for dimensions when $k_0\neq l-1$
\begin{equation*}
\frac{\dim T(k)}{(\dim  T(1))^N}=\frac{2^{2-N_0}\xi l N_1}{2^{2lN_1}}\left(1+O\left(\frac{1}{N_1}\right)\right),
\end{equation*}
and when $k_0=l-1$
\begin{equation*}
\frac{\dim T(lk_1+l-1)}{(\dim  T(1))^N}=\frac{2^{1-N_0}\xi l N_1}{2^{2lN_1}}\left(1+O\left(\frac{1}{N_1}\right)\right).
\end{equation*}

We obtain the following pointwise asymptotic for the probability measure distribution
\begin{equation}\label{Planchas}
    p^{(l)}_N(k)=\frac{\nu^{(l)}_{k_0}(\xi,N_0,\gamma)}{\sqrt{2\pi\,2lN_1}}\mathrm{exp}(2lN_1(S(\xi)-\mathrm{ln}2))\left(1+O\left(\frac{1}{N_1}\right)\right),
\end{equation}
where if $k_0\neq l-1$
\begin{equation*}
    \nu^{(l)}_{k_0}(\xi,N_0,\gamma)=\frac{16\xi^2lN_1}{\sqrt{(1-\xi^2)^{1+N_0}}}\Big(\frac{1-\xi}{1+\xi}\Big)^{\frac{\gamma l}{2}}\frac{\frac{1}{1+\xi}\Big(\frac{1+\xi}{1-\xi}\Big)^{\frac{l-k_0}{2}}-\frac{1}{1-\xi}\Big(\frac{1+\xi}{1-\xi}\Big)^{-\frac{l-k_0}{2}}}{\Big(\frac{1+\xi}{1-\xi}\Big)^{\frac{l}{2}}-\Big(\frac{1+\xi}{1-\xi}\Big)^{-\frac{l}{2}}},
\end{equation*}
and if $k_0=l-1$
\begin{equation*}
    \nu^{(l)}_{k_0}(\xi,N_0,\gamma)=\frac{8\xi^2 lN_1}{\sqrt{(1-\xi^2)^{1+N_0}}}\frac{1}{1+\xi}\Big(\frac{1-\xi}{1+\xi}\Big)^{\frac{l(1+\gamma)-1}{2}}
\end{equation*}

Similarly to the character measure, the large deviation rate function $S(\xi)$ for the Plancherel measure has its maximum at the isolated critical point $\xi_0$, i.e.
\begin{equation*}
\frac{\partial}{\partial\xi}(S(\xi)-\mathrm{ln}2)=0
\end{equation*}
has a unique solution $\xi_0=0$, where the Plancherel measure concentrates asymptotically. Thus, we have the following theorem.
\begin{theorem}\label{planchfluct}
    Let $N=2lN_1+N_0$ and $N_1\to\infty$ such that $N_0$, $l$ are fixed. Then for $k=k_1l+k_0$, random variables $k_0$ and $\alpha=\frac{k_1}{\sqrt{2N_1}}$ converge to random variables on $\{0,\ldots,l-1\}\times \mathbb{R}_{\geq 0}$ and the Plancherel probability distribution weakly converges to the tensor product of two distributions $p_1(\bullet)\otimes p_2(\bullet)$ on $\{0,\ldots,l-1\}\times \mathbb{R}_{\geq 0}$. The distribution $p_1(\bullet)$ is a discrete probability distribution on $k_0=0,1,\ldots,l-1$ with density given by
	\begin{equation*}
	p_1(k_0)=\begin{cases}
	\frac{2(l-1-k_0)}{l^2},& \textit{if $k_0=0,\ldots,l-2$},\\
	\frac{1}{l},& \textit{if $k_0=l-1$}
	\end{cases},
	\end{equation*}
	The distribution $p_2(\bullet)$ is a radial part of the Gaussian distribution on $\mathbb{R}^3$ with density given by
    \begin{equation*}
    p_2(\alpha)=\sqrt{\frac{2}{\pi}}l^{\frac{3}{2}}\alpha^2e^{-\frac{\alpha^2 l}{2}}.
    \end{equation*}
\end{theorem}
The proof of this theorem is presented in Appendix \ref{approofplanch}.

This asymptotic does not depend on $N_0$, $\gamma$. This is why we expect that the theorem holds for any sequence $\{N^{(p)}\}_{p\in\mathbb{Z}_{\geq 0}}$ with $N^{(p)}\to\infty$ as $p\to\infty$.

\subsection{The near Plancherel asymptotic of the character measure}
Now consider the character measure but assume that $t=\frac{u}{\sqrt{2N_1}}$ as $N_1\to\infty$. In this case the character measure weakly converges to a deformation of a limiting measure obtained in Theorem \ref{planchfluct}.
\begin{theorem}\label{intermsc}
    Let $N=2lN_1+N_0$ and $N_1\to\infty$ such that $N_0$, $l$ are fixed and $t=\frac{u}{\sqrt{2N_1}}$. Then for $k=k_1l+k_0$, random variables $k_0$ and $b=\frac{k_1}{\sqrt{2N_1}}$ converge to random variables on $\{0,\ldots,l-1\}\times \mathbb{R}_{\geq 0}$ and the character probability distribution weakly converges to the tensor product of two distributions $p_1(\bullet)\otimes p_2(\bullet;u)$ on $\{0,\ldots,l-1\}\times \mathbb{R}_{\geq 0}$. The distribution $p_1(\bullet)$ is a discrete probability distribution on $k_0=0,1,\ldots,l-1$ with density given by
	\begin{equation*}
	p_1(k_0)=\begin{cases}
	\frac{2(l-1-k_0)}{l^2},& \textit{if $k_0=0,\ldots,l-2$},\\
	\frac{1}{l},& \textit{if $k_0=l-1$}
	\end{cases},
	\end{equation*}
	The distribution $p_2(\bullet;u)$ is a distribution on $\mathbb{R}_{\geq 0}$ with density given by
	\begin{equation*}
	p_2(b,u)=\frac{b}{u}\sqrt{\frac{l}{2\pi}}(e^{lbu}-e^{-lbu})e^{-\frac{u^2l}{2}} e^{-\frac{b^2l}{2}}.
	\end{equation*}
\end{theorem}
The proof of this theorem is presented in Appendix \ref{proofintermsc}.
Analogously, we expect that this theorem holds for any sequence of $\{N^{(p)}\}_{p\in\mathbb{Z}_{\geq 0}}$, where $N^{(p)}\to\infty$ as $p\to\infty$.

\subsection{The Poisson asymptotic of the character measure}
There is the limit when $N\to\infty$ and $t\to\infty$ such that the random walk induced by the Markov process for the character measure is drifted with the critical slope, meaning, the limit shape $\xi_0$ is on the right boundary of the Weyl alcove. It turns out that when $t$ increases logarithmically as $N\to\infty$ the character probability distribution weakly converges to a Poisson distribution in the vicinity of the right boundary of the Weyl alcove.

The following theorem gives the pointwise asymptotic of the multiplicity function in the limit when $N\to\infty$, $k\to\infty$ such that $N-k$ is fixed.
\begin{theorem}\label{multdrift}
    Assume that $N=2lN_1+N_0$, $k=k_1l+k_0$, $N_1\to\infty$, $k_1\to\infty$, such that $l$, $N_0=0,\dots,2l-1$, $k_0=0,\dots,l-1$ remain fixed, $N-k=2s$ for some $s\in\mathbb{Z}_{\geq 0}$. Then the asymptotic of $M^{(l)}_{T(k)}(N)$ is given by
	\begin{equation*}
	M^{(l)}_{T(k)}(N)=\frac{N^\frac{N-k}{2}}{\big(\frac{N-k}{2}\big)!}\Big(1+o(1)\Big).
	\end{equation*}
\end{theorem}
The proof of this theorem is presented in Appendix \ref{poisproof1}.

Using results of the Theorem \ref{multdrift}, we obtain the following limiting probability distribution for the character measure near the right boundary.
\begin{theorem}\label{mesdrift}
    Let $N=2lN_1+N_0$ and $N_1\to\infty$, such that $N_0$, $l$ are fixed. If $t\to\infty$ such that $\theta=Ne^{-2t}$ is fixed, then for $k=k_1l+k_0$, random variable $s=\frac{N-k}{2}$ converges to a random variable on $\mathbb{Z}_{\geq 0}$ and the character probability distribution weakly converges to the Poisson distribution on $\mathbb{Z}_{\geq 0}$ with density given by
	\begin{equation*}
	p(s;\theta)=\frac{\theta^{s}e^{-\theta}}{s!}.
	\end{equation*}
\end{theorem}
The proof of this theorem is presented in Appendix \ref{poisproof2}.
This result coincides with the one for classical Lie algebra $\lie{sl}_2$.

\section{The asymptotic of probability measures for $u_q(\lie{sl}_2)$}\label{as10}
\subsection{Corresponding Markov processes}
\subsubsection{Probabilistic measures induced by Markov processes}
Let $X$ be a countable set. A {\it local Markov process} on $X$ is a collection of transition probabilities $\{\mathrm{Prob}(x\to y)\}_{x,y\in X}$, i.e. $0\leq \mathrm{Prob}(x\to y)\leq 1$, such that
\begin{itemize}
\item for each $x$ there are only finitely many $y$ such that $\mathrm{Prob}(x\to y)\neq 0$,
\item $\sum\limits_y \mathrm{Prob}(x\to y)=1$, i.e. $x$ with probability $1$ will transition to some $y$. 
\end{itemize}
Let $\{p(x)\}$ be a probability measure on $X$, i.e. for each $x\in X$, $0\leq p(x) \leq 1$ and $\sum\limits_x p(x)=1$. We will say it is {\it local}, if $p(x)\neq 0$ only for finitely many $x\in X$. 

One of the fundamental properties of Markov processes is that function 
\[
p_N(x)=(\mathrm{Prob}^Np)(x)=\sum_{x_1,x_2,\dots, x_N}p(x_1)\mathrm{Prob}(x_1\to x_2)\dots \mathrm{Prob}(x_N\to x)
\]
is a probability measure on $X$ for each $N\in\mathbb{Z}_{>0}$, $N\geq 2$. For each such $N$ it is local if $p(x)$ is local and the Markov process is local. In particular, if $p(x)=\delta_{x,x_0}$ by applying repeatedly a local Markov process to such distribution we have a sequence of probability distributions on $X$:
\[
p_{x_0}(x)=\sum_{x_1,x_2,\dots, x_{N-1}}\mathrm{Prob}(x_0\to x_1)\dots \mathrm{Prob}(x_{N-1}\to x)
\]
It turned out that character and Plancherel measures introduced earlier are exactly of this nature.

\subsubsection{Character and Plancherel measures for $U_q(\lie{sl}_2)$} 
Consider the Markov process on $\Z_{\geq 0}$ with transition probabilities determined by tensor product decomposition (\ref{A}):
\begin{equation*}
\mathrm{Prob}(n\to m)=\frac{\mathrm{ch}_{T(m)}(x)\dim\mathrm{Hom}(T(n)\otimes T(1),T(m))}{\mathrm{ch}_{T(n)}(x)\mathrm{ch}_{T(1)}(x)}
\end{equation*}
Here $x=e^t>0$ which ensures positivity of transition probabilities. The Markov property $\sum_m \mathrm{Prob}(n\to m)=1$ is guaranteed by the decomposition rule of the tensor product. The character measure on the set of tilting components of $T(1)^{\otimes N}$ can be regarded as a probability distribution on $\Z_{\geq 0}$ induced by this Markov process. 

The probability $p^{(l)}_N(k;t)$ is the result of iteration of this random process starting from the distribution supported at $k=0$, i.e. from $p^{(l)}_0(k;t)=\delta_{k,0}$:
\[
p^{(l)}_N(k;t)=\sum_{\gamma} \mathrm{Prob}(0\to 1)\mathrm{Prob}(1\to k_2)\dots \mathrm{Prob}(k_{N-1}\to k)
\]
Here the sum is taken over all lattice paths described in Section \ref{lpm}. A path $\gamma$ starts at the origin, follows the rules of elementary steps and ends at $k$ after $N+1$ steps $(k_i,k_{i+1})$.

Specializing the character measure to $t=0$ we obtain the Markov process interpretation of the Plancherel measure for $U_q(\lie{sl}_2)$. The results on the asymptotic of probability measures for $U_q(sl_2)$ can also be obtained by analyzing this Markov process. It will be done in a separate publication.

\subsubsection{Plancherel and quantum Plancherel measures for $u_q(\lie{sl}_2)$}
The decomposition rules \ref{B} define two natural Markov processes on $\{0,1,\ldots,3l-2\}$:
\begin{equation}\label{probmatrix}
\mathrm{Prob}(n\to m)=\frac{d_m \dim\mathrm{Hom}(T(n)\otimes T(1),T(m))}{d_n d_1}.
\end{equation}
Here $T(n)$ are representations $T(n)$ of $u_q(\lie{sl}_2)$ as described in Section \ref{secplanch}. For Plancherel and quantum Plancherel $d_n$ are dimensions $\dim T(n)$ and quantum dimensions $\mathrm{dim}_q T(n)$ of representation $T(n)$ correspondingly, see Sections \ref{secplanch} and \ref{qdims}.

\subsection{The asymptotic of Plancherel and quantum Plancherel measures on $u_q(\lie{sl}_2)$-modules}\label{markovuq}
The asymptotic of these measures can be computed directly from the analysis of multiplicities. But it is much easier to use Markov processes in this case.

\subsubsection{Plancherel measure}
For the small quantum group the Markov matrix (\ref{probmatrix}) is a real $3l-1\times 3l-1$ matrix
with non-negative matrix elements:
\begin{equation*}
\mathrm{Prob}(k\to k+1)=\begin{cases}
\frac{k+2}{2(k+1)},&\textit{if $0\leq k \leq l-2$}\\
1,&\textit{if $k=l-1$, $2l-1$}\\
0,&\textit{if $k=3l-2$}\\
\frac{1}{2},&\textit{if else}
\end{cases}
\end{equation*}
\begin{equation*}
\mathrm{Prob}(k\to k-1)=\begin{cases}
\frac{k}{2(k+1)},&\textit{if $0\leq k \leq l-2$}\\
0,&\textit{if $k=l-1, 2l-1$}\\
\frac{1}{2},&\textit{if else}
\end{cases}
\end{equation*}
\begin{equation*}
\mathrm{Prob}(3l-2\to l-1)=\frac{1}{2}.
\end{equation*}

It is not difficult to check explicitly that its largest eigenvalue $\lambda_{\text{max}}=1$ and that it has multiplicity $1$. Thus the stationary distribution in this case is given by components of the eigenvector corresponding to the largest eigenvalue of the Markov matrix. Straightforward computation shows that it is given by 
\begin{equation*}
(p_0,\ldots, p_{3l-2}),
\end{equation*}
where
\begin{equation*}
p_j=\begin{cases}
0,&\textit{if $k\leq l-2$}\\
\frac{2l-1-k}{l^2},&\textit{if $l\leq k\leq 2l-2$,}\\
\frac{3l-1-k}{l^2},&\textit{if $2l\leq k\leq 3l-2$,}\\
\frac{1}{2l},&\textit{if $k=l-1$ or $k=2l-1$.}
\end{cases}
\end{equation*}
It is normalized to the total probability condition
$$
\sum\limits_{j=0}^{3l-2}p_j=1.
$$

\subsubsection{Quantum Plancherel measure}
The Markov matrix for the quantum Plancherel measure is given by transition probabilities
\begin{equation*}
\mathrm{Prob}(k\to k+1)=\begin{cases}
\frac{q^{k+2}-q^{-(k+2)}}{(q^{k+1}-q^{-(k+1)})(q+q^{-1})},&\textit{if $0\leq k \leq l-3$}\\
0,&\textit{if $k=l-2$}
\end{cases}
\end{equation*}
\begin{equation*}
\mathrm{Prob}(k\to k-1)=\begin{cases}
\frac{q^{k}-q^{-k}}{(q^{k+1}-q^{-(k+1)})(q+q^{-1})},&\textit{if $0\leq k \leq l-3$}\\
1,&\textit{if $k=l-2$}
\end{cases}
\end{equation*}
It is a $l-1 \times l-1$ matrix supported on irreducible representations because quantum dimensions of other indecomposable summands in $T(1)^{\otimes N}$ vanish. The largest eigenvalue $\lambda_{\text{max}}=1$ of this matrix has multiplicity $1$. It is not difficult to compute components of the corresponding eigenvector $(p_0,\ldots, p_{l-2})$:
\begin{equation*}
p_j=-\frac{(q^{j+1}-q^{-(j+1)})^2}{2l}=\frac{2\mathrm{sin}^2\left(\frac{\pi (j+1)}{l} \right)}{l}.
\end{equation*}
It is normalized to the total probability condition
$$
\sum\limits_{j=0}^{l-2}p_j=1.
$$

\appendix
\section{Gaussian binomial coefficients}\label{gbcc}
Let $v$ be a formal parameter. The Gaussian binomial coefficient for $m \in \Z$, $r \in \mathbb{Z}_{\geq 0}$, is defined as follows 
\[ \qbin{m}{r} = \prod_{s=1}^r \frac{v^{m-s+1} - v^{-m+s-1}}{v^s-v^{-s}}  \in \Z[v, v^{-1}]. \] 
We also define for $n \geq 0$  the quantum numbers 
\[ [n] = \frac{v^n - v^{-n}}{v - v^{-1}} , \quad \quad  [n]! = \prod_{s=1}^n \frac{v^s - v^{-s}}{v-v^{-1}} .\] 
If $0 \leq r \leq m$, then 
\[ \qbin{m}{r} = \frac{[m]!}{[r]! [m-r]!} .\] 
Now let $q = e^{\frac{\pi i}{l}}$, where $l$ is an odd integer $\geq 3$, so that $q^{2l} =1$ and $q^k \neq 1$ for any $1 \leq k < 2l$.

\begin{proposition}  \label{q-bin1} 
	Let $m = lm_1 + m_0$, where $0 \leq m_0 \leq l-1$, $m_1 \in \Z$, $r= r_1 l + r_0$, where $0 \leq r_0 \leq l-1$, $r_1 \in \mathbb{Z}_{\geq 0}$. 
	For the specialization $v \rightarrow q$ we have  
	\[ \qbin{lm_1}{lr_1} =  (-1)^{(m_1+1)r_1} { m_1 \choose r_1}   .\] 
	\[ \qbin{m}{r} =    (-1)^{(m_0 r_1 - m_1 r_0)+(m_1+1)r_1} { m_1 \choose r_1}  \qbin{m_0}{r_0} .\] 
\end{proposition} 
This follows from \cite{Lus94}, 34.1.2. 

\begin{corollary}  \label{q-bin2} 
	For an integer $p \geq 0$ and the specialization $v \rightarrow q$ we have 
	\[ \frac{[lp]!}{([l]!)^p} = p! (-1)^{p(p-1)/2} . \] 
\end{corollary} 
This  follows from Proposition \ref{q-bin1} by induction on $p$. 

\begin{corollary}  \label{q-bin3} 
	\begin{enumerate} 
		\item Let $m = m_0 + l m_1$, where $0 \leq m_0 \leq l-1$, $m_1 \in \Z$. Then 
		\[ \qbin{m}{l} = (-1)^{m +1} m_1. \] 
		\item If $m, n \in \Z$ satisfy $q^m = q^n$ and $\qbin{m}{l} = \qbin{n}{l}$, then $m=n$. 
	\end{enumerate} 
\end{corollary} 
\begin{proof} 
	If $m_1 \geq 1$, then  (a) follows from Proposition \ref{q-bin1}. If $m_1=0$, then $0 \leq m \leq l-1$ and $\qbin{m}{l} =0$ for an indeterminate $v$. If $m_1 <0$, then we have 
	\[ \qbin{m_1 l + m_0 }{l} = (-1)^l \qbin{-m_1 l - m_0 + l -1}{l} = (-1)^l (-1)^{-m_1+ (l-1 -m_0)+1} (-m_1) = (-1)^{m_1 + m_0+1}m_1.  \] 
	Using the assumption that $l$ is odd, we have $(-1)^m = (-1)^{lm_1 + m_0} = (-1)^{m_1+ m_0}$, and finally 
	\[ \qbin{m_1 l + m_0 }{l} = (-1)^{m+1} m_1 .\] 
	Now consider (b). If $q^m = q^n$, then $n = m + 2lp$ for some $p \in \Z$. Then $\qbin{n}{l} = \qbin{l(m_1+2p)+m_0}{l} = (-1)^{m_1 + 2p + m_0 +1} (m_1 + 2p) = (-1)^{m_1+ m_0 +1} (m_1+2p)$, and $\qbin{m}{l} = (-1)^{m_1+ m_0 +1}m_1$, 
	so the equality $\qbin{n}{l} = \qbin{m}{l}$ implies $p=0$ and $m=n$. 
\end{proof} 

\section{Proof of Theorem \ref{th_u}}\label{prthu}
\subsection{Proof of theorem \ref{th_u} (a)}
\begin{proof}
	Consider the counting formula for lattice path in the auxiliary model \cite{PS}
	\begin{eqnarray}\label{auxmod}
	\label{mj}
	M^{j}_{(M,N)} = 2^{j-2}\big(\sum_{k=0}^{\lfloor\frac{N-(j-1)l+1}{4l}\rfloor}P_j(k) F^{(N)}_{M+4kl}+\sum_{k=0}^{\lfloor\frac{N-jl}{4l}\rfloor}P_j(k)F^{(N)}_{M-4kl-2jl} - \nonumber\\
	-\sum_{k=0}^{\lfloor\frac{N-(j+1)l+1}{4l}\rfloor}Q_j(k)F^{(N)}_{M+2l+4kl}-\sum_{k=0}^{\lfloor\frac{N-(j+2)l}{4l}\rfloor}Q_j(k)F^{(N)}_{M-4kl-2(j+1)l} \big),
	\end{eqnarray}
	where
	\begin{equation*}
	P_j(k)=\sum_{i=0}^{\lfloor\frac{j}{2}\rfloor}\binom{j-2}{2i}\binom{k-i+j-2}{j-2},\;\;\;\;\;
	Q_j(k)=\sum_{i=0}^{\lfloor\frac{j}{2}\rfloor}\binom{j-2}{2i+1}\binom{k-i+j-2}{j-2}.
	\end{equation*}
	
	Let us map part of the auxiliary lattice to the cylinder with coordinates $(M,N)$, $l-1\leq M \leq 3l-1$ by the mapping
	\begin{equation*}
	\varphi:(M,N)\to (M\,\mathrm{mod}\,2l,N)\quad\textit{for $M\geq l-1$}.
	\end{equation*}
	We will call this operation "twisting the lattice to a cylinder". The number of paths from the origin  $(0,0)$ to the point $(M,N)$ on the cylinder is equal to the number of paths from  $(0,0)$ to all pre-images $\varphi^{-1}(M,N)$ of $(M,N)$ on the auxiliary lattice. For $2l-1\leq M\leq 3l-2$ this gives
	$$
	\sum_{j=1}^{\lfloor\frac{N+l}{2l}\rfloor}M^{2j}_{(M+(2j-1)l,N)},
	$$
	and for $l-1\leq M\leq 2l-2$, we have
	$$
	\sum_{j=1}^{\lfloor\frac{N}{2l}\rfloor}M^{2j+1}_{(M+2jl,N)},
	$$
	\begin{figure}[h!]
		\centerline{\includegraphics[width=300pt]{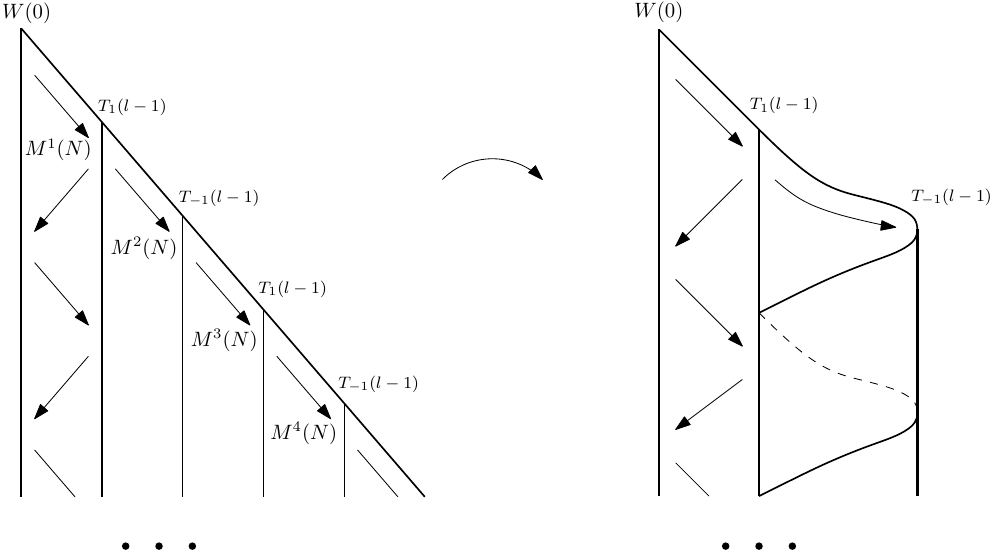}}
		\caption{Twisting of the auxiliary lattice into a cylinder.}
		\label{ttop}
	\end{figure}\\
	
	We want to prove that, after twisting into a cylinder, the number of paths computed above coincides with 
	the multiplicity of corresponding indecomposable $u_q(\lie{sl}_2)$-module in $(L(1))^{\otimes N}$. For this 
	we need to prove the identities:
	
	\begin{equation*}
	M^{(l)}_{T_{+}(m)}(N)=\sum_{j=1}^{\lfloor\frac{N+l}{2l}\rfloor}M^{2j}_{(m+(2j-1)l,N)},
	\end{equation*}
	\begin{equation*}
	M^{(l)}_{T_{-}(m)}(N)=\sum_{j=1}^{\lfloor\frac{N}{2l}\rfloor}M^{2j+1}_{(m+2jl,N)},
	\end{equation*}
	Here the left side of these identities is given by  (\ref{ind1}), (\ref{ind3}). Let us first prove the
	formula for $M^{(l)}_{T_{+}(m)}(N)$. The proof of the formula for $M^{(l)}_{T_{-}(m)}(N)$ is very similar.
	
	We will prove it by induction in $N$ with step $2l$. The base of the induction corresponds to values
	$l\leq N\leq3l-1$. We have to prove that for this values
	\begin{equation}\label{ind2}
	M^{(l)}_{T_{+}(m)}(N)=M^2_{(m+l,N)}.
	\end{equation}
	The left side of (\ref{ind2}) is
	\begin{equation*}
	M^{(l)}_{T_{+}(m)}(N)=F^{(N)}_{m+l}+F^{(N)}_{m+l-4l},
	\end{equation*}
	For the right side (\ref{mj})we obtain
	\begin{equation*}
	M^2_{(m+l,N)}=F^{(N)}_{m+l}+F^{(N)}_{m-l-2l}.
	\end{equation*}
	Clearly (\ref{ind2}) holds, which give us the base for induction.
	
	Consider now an arbitrary $N$. Define $n=\lfloor\frac{N-l}{2l}\rfloor$ and ssume that for 
	$2nl+l\leq N\leq2nl+3l-1$, the identity 
	\begin{equation*}
	M^{(l)}_{T_{+}(m)}(N)=\sum_{j=1}^{n}M^{2j}_{(m+(2j-1)l,N)},
	\end{equation*}
	holds. That is assume that for these values of $N$
	\begin{equation}\label{ded}
	\sum_{k=0}^{n-1}\big(k^2F^{(N)}_{m-l-2kl}+(k+1)^2 F^{(N)}_{m+l+2kl}\big)+n^2F^{(N)}_{m-l-2nl}=\sum_{j=1}^{n}M^{2j}_{(m+(2j-1)l,N)}.
	\end{equation}
	
	\begin{lemma}\label{lemmaqqbb} If the identity (\ref{ded}) holds for $N$ and $n=\lfloor\frac{N-l}{2l}\rfloor$, then 
		the following identity holds for the this value of $n$ and any $N'\geq N$ 
		\begin{eqnarray}\label{qwerty}
		\sum_{k=0}^{n-1}\big(k^2F^{(N^\prime)}_{m-l-2kl}+(k+1)^2 F^{(N^\prime)}_{m+l+2kl}\big)+n^2F^{(N^\prime)}_{m-l-2nl}=\\
		=\sum_{j=1}^{n}\Big(
		2^{2j-2}\big(\sum_{k=0}^{\lfloor\frac{N-(2j-1)l+1}{4l}\rfloor}P_{2j}(k) F^{(N^\prime)}_{m+4kl}+\sum_{k=0}^{\lfloor\frac{N-2jl}{4l}\rfloor}P_{2j}(k)F^{(N^\prime)}_{m-4kl-4jl} - \nonumber\\
		-\sum_{k=0}^{\lfloor\frac{N-(2j+1)l+1}{4l}\rfloor}Q_{2j}(k)F^{(N^\prime)}_{m+2l+4kl}-\sum_{k=0}^{\lfloor\frac{N-(2j+2)l}{4l}\rfloor}Q_{2j}(k)F^{(N^\prime)}_{m-4kl-2(2j+1)l} \big)\Big)\nonumber,
		\end{eqnarray}
		
	\end{lemma}
	\begin{proof}
		The RHS of (\ref{ded}) can be expanded using (\ref{auxmod}). For fixed value of $n=n_0$ coefficients of Catalan numbers $F^{(N)}_{m-l+2jl}$ and $F^{(N)}_{m-l-2jl}$ on both sides of (\ref{ded}) don't depend on $N$, so for simplicity we can write
		\begin{eqnarray}\label{lem1}
		\sum_{k=0}^{n_0-1}\big(L^{-}_kF^{(N)}_{m-l-2kl}+L^{+}_k F^{(N)}_{m+l+2kl}\big)+L^{-}_{n_0}F^{(N)}_{m-l-2n_0l}=\\
  =\sum_{k=0}^{n_0-1}\big(R^{-}_kF^{(N)}_{m-l-2kl}+R^{+}_k F^{(N)}_{m+l+2kl}\big)+R^{-}_{n_0}F^{(N)}_{m-l-2n_0l}.\nonumber
		\end{eqnarray}
		For $l\leq N\leq 2l-1$, (\ref{lem1}) becomes
		\begin{equation*}
		L^{+}_0F^{(N)}_{m+l}=R^{+}_0F^{(N)}_{m+l},
		\end{equation*}
		since all other Catalan numbers are $0$ for such values of $N$. And we get
		\begin{equation}\label{lem11}
		L^{+}_0=R^{+}_0.
		\end{equation}
		For $2l\leq N\leq 3l-1$, (\ref{lem1}) becomes
		\begin{equation*}
		L^{-}_1F^{(N)}_{m-l-2l}+L^{+}_0 F^{(N)}_{m+l}=R^{-}_1F^{(N)}_{m-l-2l}+R^{+}_0 F^{(N)}_{m+l},
		\end{equation*}
		and taking into account (\ref{lem11}), we get
		\begin{equation*}
		L^{-}_1=R^{-}_1.
		\end{equation*}
		Carrying out this procedure till the end, we get a system on $2n_0$ identities
		\begin{eqnarray*}
		L^{+}_0 &=& R^{+}_0,\\
		L^{\pm}_j &=& R^{\pm}_j,\quad j=1,\ldots, n_0-1,\\
		L^{-}_{n_0} &=& R^{-}_{n_0}.
		\end{eqnarray*}
		Since identities in the system don't depend on $N$, we see that when $n_0$ is fixed formula (\ref{ded}) should hold for $N^\prime\geq N$, which concludes the proof of the lemma.
	\end{proof}\par
	
	Now let us prove the induction step. We should prove that the induction assumption (\ref{ded}) implies the identity
    \begin{equation}\label{needtoprove}
	\sum_{k=0}^{n}\big(k^2F^{(N+2l)}_{m-l-2kl}+(k+1)^2 F^{(N+2l)}_{m+l+2kl}\big)+(n+1)^2F^{(N+2l)}_{m-l-2nl}=\sum_{j=1}^{n+1}M^{2j}_{(m+(2j-1)l,N+2l)}.
	\end{equation}
	Let us write the left side as
	\begin{eqnarray*}
	\Big(\big(\sum_{k=0}^{n-1}k^2F^{(N+2l)}_{m-l-2kl}+(k+1)^2 F^{(N+2l)}_{m+l+2kl}\big)+n^2F^{(N+2l)}_{m-l-2nl}\Big)+\nonumber \\
	+(n+1)^2 F^{(N+2l)}_{m+l+2nl}+(n+1)^2F^{(N+2l)}_{m-l-2(n+1)l}.
	\end{eqnarray*}
	For even $n$ the right side of (\ref{needtoprove}) can be written as 
	\begin{eqnarray}\label{id-proof}
	\sum_{j=1}^{n+1}M^{2j}_{(m+(2j-1)l,N+2l)}=\Big(\sum_{j=1}^{n}\Big(
	2^{2j-2}\big(\sum_{k=0}^{\lfloor\frac{N-(2j-1)l+1}{4l}\rfloor}P_{2j}(k) F^{(N+2l)}_{m+4kl}+\\
    \sum_{k=0}^{\lfloor\frac{N-2jl}{4l}\rfloor}P_{2j}(k)F^{(N+2l)}_{m-4kl-4jl}-\sum_{k=0}^{\lfloor\frac{N-(2j+1)l+1}{4l}\rfloor}Q_{2j}(k)F^{(N+2l)}_{m+2l+4kl}-\sum_{k=0}^{\lfloor\frac{N-(2j+2)l}{4l}\rfloor}Q_j(k)F^{(N+2l)}_{m-4kl-2(2j+1)l} \big)\Big) \Big)+\nonumber\\
	+\sum_{j=1,\text{$j$ odd}}^{n-1}2^{2j-2}P_{2j}(\frac{n-(j-1)}{2})\Big(F^{(N+2l)}_{m+(2j-1)l+4(\frac{n-(j-1)}{2})l}+F^{(N+2l)}_{m+(2j-1)l-4(\frac{n-(j-1)}{2})l-4jl}\Big)-\nonumber\\
	-\sum_{j=2,\text{$j$ even}}^{n}2^{2j-2}Q_{2j}(\frac{n-j}{2})\Big(F^{(N+2l)}_{m+(2j-1)l+2l+4(\frac{n-j}{2})l}+F^{(N+2l)}_{m+(2j-1)l-4(\frac{n-j}{2})l-2(2j+1)l}\Big)+\nonumber\\
	+M^{2(n+1)}_{(m+(2n+1)l,N+2l)},\nonumber
	\end{eqnarray}
	and for odd $n$ as	
	\begin{eqnarray*}
	\sum_{j=1}^{n+1}M^{2j}_{(m+(2j-1)l,N+2l)}=\Big(\sum_{j=1}^{n}\Big(
	2^{2j-2}\big(\sum_{k=0}^{\lfloor\frac{N-(2j-1)l+1}{4l}\rfloor}P_{2j}(k) F^{(N+2l)}_{m+4kl}+ \nonumber\\
	+\sum_{k=0}^{\lfloor\frac{N-2jl}{4l}\rfloor}P_{2j}(k)F^{(N+2l)}_{m-4kl-4jl}-\sum_{k=0}^{\lfloor\frac{N-(2j+1)l+1}{4l}\rfloor}Q_{2j}(k)F^{(N+2l)}_{m+2l+4kl}-\sum_{k=0}^{\lfloor\frac{N-(2j+2)l}{4l}\rfloor}Q_j(k)F^{(N+2l)}_{m-4kl-2(2j+1)l} \big)\Big) \Big)+\nonumber\\
	+\sum_{j=2,\text{$j$ even}}^{n-1}2^{2j-2}P_{2j}(\frac{n-(j-1)}{2})\Big(F^{(N+2l)}_{m+(2j-1)l+4(\frac{n-(j-1)}{2})l}+F^{(N+2l)}_{m+(2j-1)l-4(\frac{n-(j-1)}{2})l-4jl}\Big)-\nonumber\\
	-\sum_{j=1,\text{$j$ odd}}^{n}2^{2j-2}Q_{2j}(\frac{n-j}{2})\Big(F^{(N+2l)}_{m+(2j-1)l+2l+4(\frac{n-j}{2})l}+F^{(N+2l)}_{m+(2j-1)l-4(\frac{n-j}{2})l-2(2j+1)l}\Big)+\nonumber\\
	+M^{2(n+1)}_{(m+(2n+1)l,N+2l)}.
	\end{eqnarray*}
	Let us focus on even $n$. For odd $n$ the proof is similar. Using results of Lemma \ref{lemmaqqbb}
	we can simplify the identity (\ref{needtoprove}) to 
	\begin{eqnarray}\label{even}
	(n+1)^2(F^{(N+2l)}_{m+l+2nl}+F^{(N+2l)}_{m-l-2(n+1)l})=M^{2(n+1)}_{(m+(2n+1)l,N+2l)}+\\
 \sum_{j=1,\text{$j$ odd}}^{n-1}2^{2j-2}P_{2j}(\frac{n-(j-1)}{2})\Big(F^{(N+2l)}_{m+(2j-1)l+4(\frac{n-(j-1)}{2})l}+F^{(N+2l)}_{m+(2j-1)l-4(\frac{n-(j-1)}{2})l-4jl}\Big)- \nonumber \\
	\sum_{j=2,\text{$j$ even}}^{n}2^{2j-2}Q_{2j}(\frac{n-j}{2})\Big(F^{(N+2l)}_{m+(2j-1)l+2l+4(\frac{n-j}{2})l}+F^{(N+2l)}_{m+(2j-1)l-4(\frac{n-j}{2})l-2(2j+1)l}\Big)\nonumber
	\end{eqnarray}\par
	Let us write left side of this identity as 
	\begin{equation*}
	(n+1)^2( F^{(N+2l)}_{m+l+2nl}+F^{(N+2l)}_{m-l-2(n+1)l})=(n+1)^2( F^{(N+2l)}_{m+(2n+1)l}+F^{(N+2l)}_{m-(2n+3)l}).
	\end{equation*}
	Now, taking into account that (\ref{mj}) implies the identity 
	\begin{eqnarray*}
		M^{2(n+1)}_{(m+(2n+1)l,N+2l)}=2^{2n}P_{2(n+1)}(0)\Big(F^{(N+2l)}_{m+(2n+1)l}+F^{(N+2l)}_{m-(2n+3)l} \Big),
	\end{eqnarray*}
	and that
	\begin{equation*}
	F^{(N+2l)}_{m+(2j-1)l+4(\frac{n-(j-1)}{2})l}+F^{(N+2l)}_{m+(2j-1)l-4(\frac{n-(j-1)}{2})l-4jl}=F^{(N+2l)}_{m+(2n+1)l}+F^{(N+2l)}_{m-(2n+3)l},
	\end{equation*}
	\begin{equation*}
	F^{(N+2l)}_{m+(2j-1)l+2l+4(\frac{n-j}{2})l}+F^{(N+2l)}_{m+(2j-1)l-4(\frac{n-j}{2})l-2(2j+1)l}=F^{(N+2l)}_{m+(2n+1)l}+F^{(N+2l)}_{m-(2n+3)l},
	\end{equation*}
	we can write the right side of  (\ref{even}) as
	\begin{equation*}
	\Big(\sum_{j=1,\text{$j$ odd}}^{n+1}2^{2j-2}P_{2j}(\frac{n-(j-1)}{2})-\sum_{j=2,\text{$j$ even}}^{n}2^{2j-2}Q_{2j}(\frac{n-j}{2})\Big)\Big(F^{(N+2l)}_{m+(2n+1)l}+F^{(N+2l)}_{m-(2n+3)l}\Big).
	\end{equation*}
	So, we have to prove the following statement
	
	\begin{lemma}
		For even $n$ the identity
		\begin{eqnarray}\label{ci}
		(n+1)^2=\sum_{j=1,\text{$j$ odd}}^{n+1}2^{2j-2}P_{2j}(\frac{n-(j-1)}{2})-\sum_{j=2,\text{$j$ even}}^{n}2^{2j-2}Q_{2j}(\frac{n-j}{2}).
		\end{eqnarray}
		holds.
	\end{lemma}
	\begin{proof}
		Let us prove it by induction. For $n=0$ the identity holds which gives the base for induction.
		The inductive assumption for $n-1$ is
		\begin{eqnarray}\label{ci2}
		(n-1)^2=\sum_{j=1,\text{$j$ odd}}^{n-1}2^{2j-2}P_{2j}(\frac{n-1-(j-1)}{2})-\sum_{j=2,\text{$j$ even}}^{n-2}2^{2j-2}Q_{2j}(\frac{n-1-j}{2}),
		\end{eqnarray}
		We need to prove (\ref{ci}). Note that 
		\begin{equation*}
		P_{j}(k)-P_{j}(k-1)=\binom{j+2k-3}{j-3},
		\end{equation*}
		\begin{equation*}
		Q_{j}(k)-Q_{j}(k-1)=\binom{j+2k-2}{j-3}.
		\end{equation*}
		Using these identities subtract  (\ref{ci2}) from (\ref{ci})
		\begin{eqnarray*}
		4n=\Big(\sum_{j=1,\text{$j$ odd}}^{n-1}-\sum_{j=2,\text{$j$ even}}^{n-2}\Big)2^{2j-2}\binom{j+n-2}{2j-3}+\\ \nonumber
		+2^{2n}P_{2n+2}(0)-2^{2n-2}Q_{2n}(0).
		\end{eqnarray*}
		Because $P_j(0)=1$ and $Q_j(0)=j-2$, this formula simplifies to	
		\begin{equation}\label{ci3}
		n=\sum_{j=0}^{n-1}(-1)^{j+1}4^j\binom{j+n}{2j+1}.
		\end{equation}
		Denote the expression (\ref{ci3}) by $f(n)$. Using the binomial recursion
		\begin{equation*}
		\binom{n}{k}=\binom{n-1}{k}+\binom{n-1}{k-1}
		\end{equation*}
		we obtain
		\begin{equation*}
		f(n)=\sum_{j=0}^{n-1}(-1)^{j+1}4^j\binom{j+n-1}{2j+1}+\sum_{j=0}^{n-1}(-1)^{j+1}4^j\binom{j+n-1}{2j}.
		\end{equation*}
		where in the first sum the last term is zero. If we denote the last sum by $g(n)$, we have 
		\begin{equation}\label{rec1}
		f(n)=f(n-1)+g(n).
		\end{equation}
		After similar manipulations for $g(n)$, we obtain the identity
		\begin{equation}\label{rec2}
		g(n)=g(n-1)-4f(n-1).
		\end{equation}
		Formulae (\ref{rec1}), (\ref{rec2}) define the recursion for $f(n)$, $g(n)$. The induction base (\ref{ci2}), (\ref{rec1})
		is a boundary condition for this recursion.  Solving it we have
		\begin{eqnarray*}
		f(n)=(-1)^n n,\\
		g(n)=(-1)^{n}(2n-1),
		\end{eqnarray*}
		This proves (\ref{ci3}) for even $n$ and finishes the proof of our lemma.
	\end{proof}
	Thus, we proved the theorem for $T_{+}(m)$ and even values of $\lfloor\frac{N-l}{2l}\rfloor$.
	The proof of other cases is similar. 
\end{proof}

\subsection{Proof of theorem \ref{th_u} (b)}
\begin{proof}
Let us prove the second part. Here we will prove the first identity. 
The proof of others is very similar. According to the formula derived earlier 
from representation theory we have 
\begin{equation*}
m^{(l)}_{l+k_0,+}(N)  = \sum_{j \in \mathbb{Z},\, {\rm even} } m^{(l)}_{l+k_0,j}(N),
\end{equation*}
where
\begin{equation*}
m^{(l)}_{l+k_0,\pm j}(N)  = \sum_{m \geq \lfloor \frac{j}{2} \rfloor} M^{(l)}_{l+k_0+ j({\rm mod}2)l + 2ml}(N)
\end{equation*}
and
\begin{eqnarray*}
M_{k_1 l +k_0}^{(l)}(N) =   
\sum\limits_{n=0}^{\lfloor\frac{N-k}{2l}\rfloor  } F_{(k_1+2n)l + k_0}^{(N)}- \sum\limits_{n=0}^{\lfloor\frac{N-(k_1+ 2)l +k_0+2}{2l}\rfloor} F_{(k_1+2n +2)l -k_0-2}^{(N)}
\end{eqnarray*}
After direct substitution we obtain
\begin{eqnarray*}
&m^{(l)}_{l+k_0,+}(N)=\sum\limits_{j \in \mathbb{Z},\, {\rm even} } \sum\limits_{m \geq \lfloor \frac{|j|}{2} \rfloor} M^{(l)}_{l+k_0+ |j|({\rm mod}2)l + 2ml}(N)=\sum\limits_{j \in \mathbb{Z},\, {\rm even} } \sum\limits_{m \geq 0} M^{(l)}_{l+k_0+|j|l+2ml}(N)=\\
&=\sum\limits_{j \in \mathbb{Z},\, {\rm even} } \sum\limits_{m \geq 0}\Big(\sum\limits_{n\geq 0} F_{(|j|+1+2m+2n)l + k_0}^{(N)} 
- \sum\limits_{n\geq 0} F_{(|j|+1+2m+2n+2)l -k_0-2}^{(N)}\Big)=\\
&=\sum\limits_{j \in \mathbb{Z},\, {\rm even} }\Big( \sum\limits_{m \geq 0}(m+1)F_{(|j|+1+2m)l + k_0}^{(N)} 
-  \sum\limits_{m \geq 1} m F_{(|j|+1+2m)l -k_0-2}^{(N)}\Big)=\\
&=\sum\limits_{m \geq 0}(m+1)^2 F_{(1+2m)l + k_0}^{(N)} + \sum\limits_{m \geq 1} m^2 F_{-(1+2m)l+k_0}^{(N)}.
\end{eqnarray*}
Here we used the identities
\begin{equation*}
-F_{kl -k_0-2}^{(N)}=F_{-kl+k_0}^{(N)},
\end{equation*}
\begin{equation*}
\sum_{j \in \mathbb{Z},\, {\rm even}, |j|+2\tilde{m}=2m }(\tilde{m}+1)=(m+1)+2\sum_{j=0}^{m-1}(m-j)=(m+1)^2,
\end{equation*}
\begin{equation*}
\sum_{j \in \mathbb{Z},\, {\rm even}, |j|+2\tilde{m}=2m }\tilde{m}=m+2\sum_{j=1}^{m}(m-j)=m^2.
\end{equation*}
Thus the formula for $m^{(l)}_{l+k_0,+}(N)$ derived earlier from representation theory matches the one derived from the analysis of lattice paths. 
\end{proof}

\section{Proof of Theorems \ref{multasymptmain}, \ref{theomch}, \ref{planchfluct}}\label{prthu2}
\subsection{Proof of Theorem \ref{multasymptmain}}\label{apmultas}
\begin{proof}
    Consider the case $k_0\neq l-1$. Since $N\equiv k\,\mathrm{mod}\,2$, we can write $k_1=2s+\gamma$ where $s\to\infty$.

    The asymptotic of a $j$th term from the first sum in (\ref{p-numbers}) is
	\begin{eqnarray*}
		\frac{(2lN_1+N_0)!(k_1l+k_0+2jl+1)}{\big(\frac{2lN_1+N_0-k_1l-k_0-2jl}{2}\big)! \big(\frac{2lN_1+N_0+k_1l+k_0+2jl}{2}+1\big)!}=\frac{(2\xi lN_1+\gamma l+k_0+2jl+1)}{\frac{2lN_1(1+\xi)+\gamma l+k_0+2jl}{2}+1}\times\\
		\times\prod_{m=1}^{N_0}(2lN_1+m)\frac{(2lN_1)!}{\big(\frac{1-\xi-\epsilon_-}{2}2lN_1\big)! \big(\frac{1+\xi+\epsilon_+}{2}2lN_1\big)!}=
	\end{eqnarray*}
	Here 
    $$
    \frac{k_0+\gamma l+2jl\pm N_0}{2lN_1}=\epsilon_\pm
    $$ 
    and $j$ is fixed as $N_1, k_1\to \infty$. Using Stirling formula and the Taylor expansion $(x\pm\frac{\epsilon}{2})\mathrm{ln}(x\pm\frac{\epsilon}{2})=x\mathrm{ln}x\pm\frac{\epsilon}{2}\mathrm{ln}x\pm\frac{\epsilon}{2}+O(\epsilon^2)$ we get
	\begin{eqnarray*}
		&=\frac{2\xi(2lN_1)^{N_0}}{1+\xi} \sqrt{\frac{4}{2\pi 2lN_1(1-\xi)(1+\xi)}}\mathrm{exp}\Big((2lN_1)\mathrm{ln}(2lN_1)- \left(\frac{1-\xi -\epsilon_-}{2}2lN_1\right)\mathrm{ln}\left(\frac{1-\xi-\epsilon_-}{2}2lN_1\right)-\\
		&-\left(\frac{1+\xi+\epsilon_+}{2}2lN_1\right)\mathrm{ln}\left(\frac{1+\xi +\epsilon_+}{2}2lN_1\right)-2lN_1+\frac{1-\xi -\epsilon_-}{2}2lN_1+\frac{1+\xi +\epsilon_+}{2}2lN_1\Big)\left(1+O\left(\frac{1}{N_1}\right)\right)=\\
		&=\frac{2^{1+N_0}\xi}{\sqrt{\pi lN_1 (1-\xi^2)^{1+N_0}}} e^{2lN_1\big(- (\frac{1-\xi}{2})\mathrm{ln}(\frac{1-\xi}{2})-(\frac{1+\xi}{2})\mathrm{ln}(\frac{1+\xi }{2})\big)}\frac{1}{1+\xi}\Big(\frac{1-\xi}{1+\xi}\Big)^{\frac{k_0+\gamma l}{2}+jl}\left(1+O\left(\frac{1}{N_1}\right)\right).
	\end{eqnarray*}

	Similarly, the asymptotic of the $j$-th term in the second sum of (\ref{p-numbers})is
	\begin{eqnarray*}
		-\frac{2^{1+N_0}\xi}{\sqrt{\pi lN_1 (1-\xi^2)^{1+N_0}}} e^{2lN_1\big(- (\frac{1-\xi}{2})\mathrm{ln}(\frac{1-\xi}{2})-(\frac{1+\xi}{2})\mathrm{ln}(\frac{1+\xi}{2})\big)}\frac{1}{1-\xi}\Big(\frac{1-\xi}{1+\xi}\Big)^{-\frac{k_0-\gamma l}{2}+jl}\left(1+O\left(\frac{1}{N_1}\right)\right).
	\end{eqnarray*}

	Note that since $N\geq k_1l+k_0$, we have $\xi\in(0,1)$. Combining previous formulae 
	and the fact that as $N\to\infty$ the contributions from large $j$ are exponentially suppressed,
	we obtain the following asymptotic: 
	\begin{eqnarray*}
		M^{(l)}_{T(k)}(N)=\frac{2^{1+N_0}\xi}{\sqrt{\pi lN_1 (1-\xi^2)^{1+N_0}}}e^{2lN_1S(\xi)}\times\\
		\times\Big(\frac{1}{1+\xi}\sum_{j=0}^\infty\Big(\frac{1-\xi}{1+\xi}\Big)^{\frac{k_0+\gamma l}{2}+jl}-\frac{1}{1-\xi}\sum_{j=1}^\infty\Big(\frac{1-\xi}{1+\xi}\Big)^{-\frac{k_0-\gamma l}{2}+jl} \Big)\left(1+O\left(\frac{1}{N_1}\right)\right)=\\
		=\frac{2^{1+N_0}\xi}{\sqrt{\pi lN_1 (1-\xi^2)^{1+N_0}}}e^{2lN_1S(\xi)}\Big(\frac{1-\xi}{1+\xi}\Big)^{\frac{\gamma l}{2}}\times\\
        \times\Big(\frac{1}{1+\xi}\Big(\frac{1-\xi}{1+\xi}\Big)^{\frac{k_0}{2}}-\frac{1}{1-\xi}\Big(\frac{1-\xi}{1+\xi}\Big)^{-\frac{k_0}{2}+l} \Big)\sum_{j=0}^\infty\Big(\frac{1-\xi}{1+\xi}\Big)^{jl}\left(1+O\left(\frac{1}{N_1}\right)\right)
	\end{eqnarray*}
	Summing up the geometric series
	\begin{equation*}
		\sum_{j=0}^\infty\Big(\frac{1-\xi}{1+\xi}\Big)^{jl}=\frac{1}{1-\Big(\frac{1-\xi}{1+\xi}\Big)^{l}},
	\end{equation*}
	we arrive to the result 
	\begin{eqnarray*}
		&M^{(l)}_{T(k)}(N)=\frac{2^{1+N_0}\xi}{\sqrt{\pi lN_1 (1-\xi^2)^{1+N_0}}}e^{2lN_1S(\xi)}\Big(\frac{1-\xi}{1+\xi}\Big)^{\frac{\gamma l}{2}}\frac{\frac{1}{1+\xi}\Big(\frac{1+\xi}{1-\xi}\Big)^{\frac{l-k_0}{2}}-\frac{1}{1-\xi}\Big(\frac{1+\xi}{1-\xi}\Big)^{-\frac{l-k_0}{2}}}{\Big(\frac{1+\xi}{1-\xi}\Big)^{\frac{l}{2}}-\Big(\frac{1+\xi}{1-\xi}\Big)^{-\frac{l}{2}}}\left(1+O\left(\frac{1}{N_1}\right)\right).
	\end{eqnarray*}

	Multiplicity for the case $k_0=l-1$ consists of just one term, asymptotic of which has already been found, so we get
	\begin{eqnarray*}
		M^{(l)}_{T(lk_1+l-1)}(N)=\frac{2^{1+N_0}\xi}{\sqrt{\pi lN_1(1-\xi^2)^{1+N_0}}}e^{2lN_1S(\xi)}\frac{1}{1+\xi}\Big(\frac{1-\xi}{1+\xi}\Big)^{\frac{l(1+\gamma)-1}{2}}\left(1+O\left(\frac{1}{N_1}\right)\right).
	\end{eqnarray*}
\end{proof}

\subsection{Proof of Theorem \ref{theomch}}\label{approofchar}
\begin{proof}
    Here we outline the proof of the theorem omitting necessary estimates.

    We have to prove that for every continuous bounded test function $h(\alpha,k_0)$, $\alpha\in\mathbb{R}$, $k_0=0,1,\ldots,l-1$, as $N_1\to\infty$ with the assumptions of the Theorem \ref{theomch} the sum
    $$
    \sum\limits_{k=k_1l+k_0} p^{(l)}_N(k;t)h\left(\frac{k_1-2N_1 \xi_0}{\sqrt{2N_1}},k_0\right)
    $$
    converges to
    \begin{equation}\label{conv}
    \sum\limits_{k_0=0}^{l-1}\int\limits_{-\infty}^{+\infty}p_1(k_0;t)p_2(\alpha;t)h(\alpha,k_0)d\alpha.
    \end{equation}
    Now we expand formula (\ref{charas}) in the vicinity of $\xi=\xi_0$, where $\xi=\xi_0+\frac{\alpha}{\sqrt{2N_1}}$. Note that
    \begin{equation*}
		S(\xi_0)-\mathrm{ln}(e^t+e^{-t})+t\xi_0=0,
	\end{equation*}
	\begin{equation*}
		1-\xi_0=\frac{2e^{-t}}{e^t+e^{-t}},\quad 1+\xi_0=\frac{2e^{t}}{e^t+e^{-t}}.
	\end{equation*}
    Hence, we obtain
    \begin{equation*}
		p^{(l)}_N(\alpha,k_0;t)=v_{k_0}\sqrt{\frac{1}{\pi lN_1}}\mathrm{cosh}(t)\mathrm{exp}\left({-\frac{\alpha^2l(\mathrm{cosh}(t))^2}{2}}\right)\left(1+o\left(1\right)\right),
	\end{equation*}
	where
	\begin{equation*}
		v_{k_0}=\begin{cases}
			\frac{(e^{t(k_0+1)}+e^{-t(k_0+1)})(e^{t(l-1-k_0)}-e^{-t(l-1-k_0)})}{e^{tl}-e^{-tl}},& \textit{if $k_0=0,\ldots,l-2$},\\
			1,& \textit{if $k_0=l-1$}.
		\end{cases}
	\end{equation*}
	This expression can be written as
	\begin{equation*}
		p^{(l)}_N(\alpha,k_0,t)=p_1(k_0;t)p_2(\alpha;t)\sqrt{\frac{2}{N_1}}\left(1+o\left(1\right)\right),
	\end{equation*}
	where
	\begin{equation*}
		p_1(k_0;t)=\begin{cases}
			\frac{(e^{t(k_0+1)}+e^{-t(k_0+1)})(e^{t(l-1-k_0)}-e^{-t(l-1-k_0)})}{l(e^{tl}-e^{-tl})},& \textit{if $k_0=0,\ldots,l-2$},\\
			\frac{1}{l},& \textit{if $k_0=l-1$}
		\end{cases},
	\end{equation*}
	and
	\begin{equation*}
		p_2(\alpha;t)=\sqrt{\frac{l}{2\pi}}\mathrm{cosh}(t)\mathrm{exp}\left({-\frac{\alpha^2l}{2}\left(\mathrm{cosh}(t)\right)^2}\right),
	\end{equation*}
    which gives us density functions $p_1(k_0;t)$ and $p_2(\alpha;t)$ as in the statement of the Theorem \ref{theomch}.

    Thus, we have
    \begin{eqnarray}\label{rsum}
    \sum\limits_{k} p^{(l)}_N(k;t)h\left(\frac{k_1-2N_1 \xi_0}{\sqrt{2N_1}},k_0\right)\to\\
    \to\sum\limits_{k_0,j}p_1(k_0;t)p_2(\alpha_j;t)\sqrt{\frac{2}{N_1}}h(\alpha_j,k_0)(1+o(1)).\nonumber
    \end{eqnarray}
    Here summation over $j$ corresponds to the summation over $k_1$ in the vicinity of the critical point $\xi_0$. Because $k_1$ changes in the increments of $2$ ($N\equiv k\,\mathrm{mod}\, 2$) and $k_1=2N_1\xi_0+\sqrt{2N_1}\alpha$, in the Riemann sum in the right hand side of (\ref{rsum}) we have $\Delta k_1=\sqrt{\frac{N_1}{2}}\Delta\alpha$. Hence, we obtain the result (\ref{conv}).
\end{proof}

\subsection{Proof of Theorem \ref{planchfluct}}\label{approofplanch}
\begin{proof}
    Outline of the proof of the Theorem \ref{planchfluct} repeats the ideas of the proof of the Theorem \ref{theomch}. Expanding formula (\ref{Planchas}) in the vicinity of the critical point, where $\xi=\frac{\alpha}{\sqrt{2N_1}}$, we obtain
    \begin{equation*}
	p^{(l)}_N(\alpha,k_0)=p_1(k_0)p_2(\alpha)\sqrt{\frac{2}{N_1}},
	\end{equation*}
	where
	\begin{equation*}
	p_1(k_0)=\begin{cases}
	\frac{2(l-1-k_0)}{l^2},& \textit{if $k_0=0,\ldots,l-2$},\\
	\frac{1}{l},& \textit{if $k_0=l-1$}
	\end{cases},
	\end{equation*}
	and
	\begin{equation*}
	p_2(\alpha)=\sqrt{\frac{2}{\pi}}l^{\frac{3}{2}}\alpha^2 e^{-\frac{\alpha^2l}{2}}(1+o(1)),
	\end{equation*}
    which gives us density functions $p_1(k_0)$ and $p_2(\alpha)$ as in the statement of the Theorem \ref{planchfluct}.

    Thus, for every continuous bounded test function $h(\alpha,k_0)$, $\alpha\in\mathbb{R}_{>0}$, $k_0=0,1,\ldots,l-1$, as $N_1\to\infty$ we have
    \begin{eqnarray}\label{rsumplanch}
    \sum\limits_{k} p^{(l)}_N(k)h\left(\frac{k_1}{\sqrt{2N_1}},k_0\right)\to\sum\limits_{k_0,j}p_1(k_0)p_2(\alpha_j)\sqrt{\frac{2}{N_1}}h(\alpha_j,k_0)(1+o(1)).\nonumber
    \end{eqnarray}
    Here summation over $j$ corresponds to the summation over $k_1$ in the vicinity of the critical point $\xi_0=0$. Because $k_1$ changes in the increments of $2$ ($N\equiv k\,\mathrm{mod}\, 2$) and $k_1=\sqrt{2N_1}\alpha$, in the Riemann sum in the right hand side of (\ref{rsumplanch}) we have $\Delta k_1=\sqrt{\frac{N_1}{2}}\Delta\alpha$. Hence, we obtain
    \begin{eqnarray*}
    \sum\limits_{k=k_1l+k_0} p^{(l)}_N(k)h\left(\frac{k_1}{\sqrt{2N_1}},k_0\right)\to\sum\limits_{k_0=0}^{l-1}\int\limits_{-\infty}^{+\infty}p_1(k_0)p_2(\alpha)h(\alpha,k_0)d\alpha.
    \end{eqnarray*}
\end{proof}

\section{Proof of Theorem \ref{intermsc}}\label{proofintermsc}
\begin{proof}
	For $k_0\neq l-1$ we get
	\begin{eqnarray*}
		\frac{\mathrm{ch}_{T(k)}(e^t)}{(\mathrm{ch}_{T(1)}(e^t))^N}=\frac{e^{lbu}-e^{-lbu}}{u}\sqrt{2N_1}e^{-\frac{u^2l}{2}}2^{-2lN_1-N_0}\left(1+O\left(\frac{1}{N_1}\right)\right),
	\end{eqnarray*}
	and for $k_0=l-1$
	\begin{equation*}
	\frac{\mathrm{ch}_{T(lk_1+l-1)}(e^t)}{(\mathrm{ch}_{T(1)}(e^t))^N}=\frac{e^{lbu}-e^{-lbu}}{2u}\sqrt{2N_1}e^{-\frac{u^2l}{2}}2^{-2lN_1-N_0}\left(1+O\left(\frac{1}{N_1}\right)\right),
	\end{equation*}
	so, expanding (\ref{mass}) in the vicinity of the critical point, where $\xi=\frac{b}{\sqrt{2N_1}}$, for the considered pointwise asymptotic of the character measure we get a product of two densities
	\begin{equation*}
	p^{(l)}_N(b,k_0;u)=p_1(k_0)p_2(b;u)\sqrt{\frac{2}{N_1}}
	\end{equation*}
	and
	\begin{equation*}
	p_1(k_0)=\begin{cases}
	\frac{2(l-1-k_0)}{l^2},& \textit{if $k_0=0,\ldots,l-2$},\\
	\frac{1}{l},& \textit{if $k_0=l-1$}
	\end{cases},
	\end{equation*}
	and
	\begin{equation*}
	p_2(b;u)=\frac{b}{u}\sqrt{\frac{l}{2\pi}}(e^{lbu}-e^{-lbu})e^{-\frac{u^2l}{2}} e^{-\frac{b^2l}{2}}(1+o(1)),
	\end{equation*}
	which is a deformation of the density obtained in Theorem \ref{planchfluct}.

    Weak convergence follows from the same argument as in Appendix \ref{approofplanch}.
\end{proof}

\section{Proof of Theorems \ref{multdrift} and \ref{mesdrift}}\label{proofdrift}
\subsection{Proof of Theorem \ref{multdrift}}\label{poisproof1}
\begin{proof}
	Note that $N=k+2a$ for some $a$ finite nonnegative integer, so $N_0-k_0+l(2s+\gamma)=2a$ and is divisible by $2$. At most only $2s+1$ terms will appear in $M^{(l)}_{T(k)}(N)$, their asymptotic is given by
	\begin{eqnarray*}
		\frac{N!(k_1l+k_0+2jl+1)}{\big(\frac{N-k_1l-k_0-2jl}{2}\big)! \big(\frac{N+k_1l+k_0+2jl}{2}+1\big)!}=\frac{(2N_1l-l(2s+\gamma)+k_0+2jl+1)\prod_{m=1}^{N_0}(2lN_1+m)}{\big(\frac{N_0-k_0+l(2s+\gamma)-2jl}{2}\big)!(2lN_1+\frac{N_0+k_0-l(2s+\gamma)+2jl}{2}+1\big)}\times \\
  \times\frac{(2lN_1)!}{\big(2lN_1+\frac{N_0+k_0-l(2s+\gamma)+2jl}{2}\big)!}\sim\frac{(2N_1l)^{N_0}}{\big(\frac{N_0-k_0+l(2s+\gamma)-2jl}{2}\big)!}\frac{\sqrt{4\pi N_1l}}{\sqrt{2\pi\big(2lN_1+\frac{N_0+k_0-l(2s+\gamma)+2jl}{2}\big)}}\times\\
  \times e^{2lN_1 \mathrm{ln}(2lN_1)-2lN_1-\big(2lN_1+\frac{N_0+k_0-l(2s+\gamma)+2jl}{2}\big)\mathrm{ln}\big(2lN_1+\frac{N_0+k_0-l(2s+\gamma)+2jl}{2}\big)+\big(2lN_1+\frac{N_0+k_0-l(2s+\gamma)+2jl}{2}\big)}\sim\\
		\sim \frac{(2lN_1)^{\frac{N_0-k_0+l(2s+\gamma)-2jl}{2}}}{\big(\frac{N_0-k_0+l(2s+\gamma)-2jl}{2}\big)!}\sim \frac{N^{\frac{N-k-2jl}{2}}}{\big(\frac{N-k-2jl}{2}\big)!},
	\end{eqnarray*}
	where we used that $(x-\epsilon)\mathrm{ln}(x-\epsilon)=x\mathrm{ln}x-\epsilon \mathrm{ln}x-\epsilon$.
	Similarly
	\begin{eqnarray*}
		\frac{N!(-k_1l+k_0-2jl+1)}{\big(\frac{N+k_1l-k_0+2jl}{2}\big)! \big(\frac{N-k_1l+k_0-2jl}{2}+1\big)!}\sim-\frac{(2lN_1)^{\frac{N_0+k_0+l(2s+\gamma)-2jl}{2}+1}}{\big(\frac{N_0+k_0+l(2s+\gamma)-2jl}{2}+1\big)!}.
	\end{eqnarray*}
	We see that the very first term is dominant. For $k_0=l-1$ only one term appears, giving the same result, so we get
	\begin{equation*}
	M^{(l)}_{T(k)}(N)=\frac{N^\frac{N-k}{2}}{\big(\frac{N-k}{2}\big)!}\Big(1+o(1)\Big).
	\end{equation*}
\end{proof}

\subsection{Proof of Theorem \ref{mesdrift}}\label{poisproof2}
\begin{proof}
	In such regime for $k_0\neq l-1$ we have
	\begin{eqnarray*}
		\frac{\mathrm{ch}_{T(k)}(e^t)}{(\mathrm{ch}_{T(1)}(e^t))^N}\sim\frac{e^{t(k_1l+k_0-N)}(1+e^{-2t(k_0+1)})(1-e^{-2tk_1l})}{(1-e^{-2t})(1+e^{-2t})^N}\sim e^{-t(N-k)}e^{-\theta},
	\end{eqnarray*}
	and similarly for $k_0=l-1$, we have
	\begin{eqnarray*}
		\frac{\mathrm{ch}_{T(k)}(e^t)}{(\mathrm{ch}_{T(1)}(e^t))^N}\sim\frac{e^{t(k_1l+l-1-N)}(1-e^{-2t(k_1+1)l})}{(1-e^{-2t})(1+e^{-2t})^N}\sim e^{-t(N-(l-1))}e^{-\theta},
	\end{eqnarray*}
	so the expression is the same for both $k_0\neq l-1$ and $k_0= l-1$. Using Theorem \ref{multdrift}, we obtain
	\begin{equation*}
	p(s;\theta)=\frac{(Ne^{-2t})^{\frac{N-k}{2}}e^{-\theta}}{\big(\frac{N-k}{2}\big)!}=\frac{\theta^{s}e^{-\theta}}{s!}.
	\end{equation*}
\end{proof}

\end{document}